\documentclass[reqno]{amsart}
\usepackage{newcent}       % selects Times Roman as basic font
\usepackage{helvet}         % selects Helvetica as sans-serif font
\usepackage{courier}        % selects Courier as typewriter font
\usepackage{amsfonts}
\usepackage{amsfonts,amssymb,amsmath}
\usepackage[latin1]{inputenc}
\usepackage{color}
\usepackage{graphicx}
\usepackage{amsmath}

\usepackage{mathrsfs}

\usepackage[mathscr]{euscript}

\usepackage[sans]{dsfont}

\usepackage{tikz}

\makeatletter
\@addtoreset{equation}{section}
\makeatother

\renewcommand\thefigure{\thesection.\@arabic\c@figure}
\renewcommand\thetable{\thesection.\@arabic\c@table}

\newtheorem{theorem}{Theorem}[section]
\newtheorem{lemma}[theorem]{Lemma}
\newtheorem{proposition}[theorem]{Proposition}
\newtheorem{corollary}[theorem]{Corollary}

\newtheorem{remark}[theorem]{Remark}

\newcommand{\mc}[1]{{\mathcal #1}}
\newcommand{\mf}[1]{{\mathfrak #1}}
\newcommand{\mb}[1]{{\mathbf #1}}
\newcommand{\bb}[1]{{\mathbb #1}}
\newcommand{\ms}[1]{{\mathscr #1}}
\newcommand{\bs}[1]{{\boldsymbol #1}}

\newcommand{\p}{\partial}

\newcommand{\<}{\langle}
\renewcommand{\>}{\rangle}

\title[Nonequilibrium fluctuations of weakly asymmetric exclusion
processes] {Nonequilibrium fluctuations of one-dimensional boundary
  driven weakly asymmetric exclusion processes}

\author{Patr\'{\i}cia Gon\c{c}alves} \address{Departamento de
  Matem\'atica, PUC-RIO, Rua Marqu\^es de S\~ao Vicente, no. 225,
  22453-900, Rio de Janeiro, Rj-Brazil and CMAT, Centro de
  Matem\'atica da Universidade do Minho, Campus de Gualtar, 4710-057,
  Braga, Portugal} \curraddr{} \email{patricia@mat.puc-rio.br}

\author{Claudio Landim}
\address{\noindent IMPA, Estrada Dona Castorina 110, CEP 22460 Rio de
  Janeiro, Brasil and CNRS UMR 6085, Universit\'e de Rouen, Avenue de
  l'Universit\'e, BP.12, Technop\^ole du Madril\-let, F76801
  Saint-\'Etienne-du-Rouvray, France.}
\email{landim@impa.br}

\author{Aniura Milan\'es}
\address{Departamento de Matem\'atica, ICEx, UFMG,
Campus Pampulha, CEP 31270-901, Belo Horizonte, Brasil.}
\email{aniura@mat.ufmg.br}

\begin{document}

\begin{abstract}
  We prove the nonequilibrium fluctuations of one-dimensional,
  boundary driven, weakly asymmetric exclusion processes through a
  microscopic Cole-Hopf transformation.
\end{abstract}

\maketitle
\section{Introduction}

Nonequilibrium fluctuations of interacting particle systems around the
hydrodynamic limit is one of the main open problems in the field. It
has only been derived for few one-dimensional dynamics and no progress
has been made in the last twenty years. We refer to the last section
of \cite[Chapter 11]{kl} for references and an historical account.

We examine in this article the dynamical nonequilibrium fluctuations
of one-dimensional weakly asymmetric exclusion processes in contact
with reservoirs. In a future work, following the strategy presented in
\cite{lmo1} for the symmetric simple exclusion process, we use the
results presented here to prove the stationary fluctuations of the
density field.

The motivations are twofold. On the one hand, the investigation of the
steady states of boundary driven interacting particle systems has
attracted a lot of attention in these last fifteen years, mainly after
\cite{dls, bdgjl02}. The density fluctuations at the steady state is
an important part of the theory and it can only be seized through the
dynamical nonequilibrium fluctuations \cite{lmo1}. On the other hand,
several published results \cite{delo1} still wait for rigorous proofs.

Denote by $\mu^N_{\rm ss}$ a stationary state of a one-dimensional
weakly asymmetric exclusion processes in contact with reservoirs.  The
stationary density fluctuation field, denoted by $\bb Y_N$, acts on
smooth functions $H: [0,1]\to\bb R$ as
\begin{equation*}
\bb Y_N (H)\;=\; \frac 1{\sqrt N} \sum_{k=1}^{N-1} H(k/N) [\eta_k -
\theta_N(k)]\;, 
\end{equation*}
where $\eta$ represents a configuration and $\theta_N(k)=
E_{\mu^N_{\rm ss}}[\eta_k]$. Not much information is available on
$\theta_N(k)$, besides discrete difference equations which involve
second-order covariance terms. It follows from Theorem \ref{s2.1}
below and some straightforward arguments, presented in \cite{lmo1} in
the case of boundary driven symmetric simple exclusion processes, that
we may replace $\theta_N(k)$ by $\bar \rho(k/N)$ in the definition of
the density fluctuation field, where $\bar \rho$ is the solution of
the stationary hydrodynamic equation, provided
\begin{equation*}
\frac {1 }{\sqrt{N}} \sum_{k=1}^{N-1} H(k/N)
\big\{\bar \rho(k/N) - \theta_N(k) \big\} 
\end{equation*}
is uniformly bounded. Note that we do not need to prove that this
expression vanishes in the limit, as one would expect from the
definition of the density fluctuation field, but just that it is
uniformly bounded.

The proof of the nonequilibrium density fluctuations we present here
relies on a microscopic Cole-Hopf transformation introduced by
G\"artner \cite{g} to investigate the hydrodynamic behavior of weakly
asymmetric exclusion processes on $\bb Z$, and used by Dittrich and
G\"artner \cite{dg} to prove the nonequilibrium fluctuations of the
same models.

As in PDE, the microscopic Cole-Hopf transformation turns a nonlinear
problem involving local functions into a linear one. For this reason,
it permits to avoid proving a nonequilibrium Boltzmann-Gibbs principle
\cite[Section 11.1]{kl}, introduced by H. Rost \cite{br1}, which is
the main technical difficulty in the proof of density fluctuations.

The proof of the nonequilibrium fluctuations relies on sharp estimates
of the moments of the microscopic Cole-Hopf variables, and on sharp
estimates of the fundamental solution of initial-boundary value
semi-discrete linear partial differential equations. These results are
presented in the last two sections of this article. The bounds on the
fundamental solutions are derived in a similar way as
hypercontractivity is proven for ergodic Markov chains.

\section{Notation and Results}

\subsection{The model}
\label{bdep}

Fix $E>0$, $\alpha$, $\beta$ in $(0,1)$ and $N\ge 1$. Denote by
$\{\eta^N_t : t\ge 0\}$, the speeded-up, one-dimensional, boundary
driven, weakly asymmetric simple exclusion process with state space
$\Sigma_N=\{0,1\}^{\{1, \dots, N-1\}}$. The configurations of the
state space are denoted by the symbol $\eta$, so that $\eta(j)=1$ if
site $j$ is occupied for the configuration $\eta$ and $\eta(j)=0$ if
site $j$ is empty. The infinitesimal generator of the Markov process
is denoted by $\mathcal{L}_N$ and acts on functions $f: \Sigma_N\to\bb
R$ as
\begin{equation*}
(\mathcal{L}_N f)(\eta)\;=\;  N^2 \sum_{j=0}^{N-1} c_{j,j+1}(\eta)
\, \{ f(\sigma^{j,j+1} \eta)-f(\eta)\} \;,
\end{equation*}
where, for $1\le j\le N-2$,
\begin{equation*}
\begin{split}
&  c_{j,j+1}(\eta) \;=\; \Big( 1 + \frac E N \Big) \eta(j)
\, [1-\eta(j+1)] \; +\; \eta(j+1)\, [1-\eta(j)] \;, \\
&\quad c_{0,1}(\eta) \;=\; \Big( 1 + \frac E N \Big) \, \eta(0) \, [1-\eta(1)]
+  \eta(1) \, [1-\eta(0)] \;, \\
&\qquad c_{N-1,N}(\eta) \;=\; \Big( 1 + \frac E N \Big)\, \eta(N-1)
\, [1-\eta(N)] \;+\; \eta(N)\, [1-\eta(N-1)] \;,
\end{split}
\end{equation*}
with the convention, adopted throughout the article, that
\begin{equation}
\label{01}
\eta(0) \;=\; \alpha\;, \quad \eta(N) \;=\; \beta\;.
\end{equation}
In these formulas, $\sigma^{j,j+1} \eta$, $1\le j\le N-2$, is the
configuration obtained from $\eta$ by exchanging the occupation
variables $\eta(j)$, $\eta(j+1)$, 
\[
(\sigma^{j,j+1} \eta)(k) =
\begin{cases}
\eta(j+1),& k=j\;,\\
\eta(j), & k=j+1\;,\\
\eta(k), & k\neq j,j+1\;,\\
\end{cases}
\]
while $\sigma^{0,1} \eta = \sigma^{1}\eta$, $\sigma^{N-1,N} \eta =
\sigma^{N-1}\eta$ are the configurations obtained by flipping the
occupation variables $\eta(1)$, $\eta(N-1)$, respectively,
\[
(\sigma^j\eta)(k) =
\begin{cases}
\eta(k),& k\neq j\;, \\
1-\eta(k), & k=j\,.\\
\end{cases}
\]

\subsection{Hydrodynamic limit}
\label{HL}

Let $D(\mathbb{R}_+,\Sigma_N)$ be the space of $\Sigma_N$-valued
functions which are right continuous with left limits, endowed with
the Skorohod topology. For a probability measure $\mu_N$ on
$\Sigma_N$, denote by $\mathbb {P}_{\mu_N}$ the measure on
$D(\mathbb{R}_+,\Sigma_N)$ induced by the Markov process $\eta^N_t$
with initial distribution $\mu_N$. We represent by
$\mathbb{E}_{\mu_N}$ the expectation with respect to $\mathbb
{P}_{\mu_N}$ and by $E_{\mu_N}$ the expectation with respect to
$\mu_N$.

Let $\pi_t^N(du)$, $t\ge 0$, be the positive random measure on $[0,1]$
obtained by rescaling space by $N^{-1}$ and by assigning mass $N^{-1}$
to each particle:
\[
\pi_t^N(dx) \;=\; \frac{1}{N} \sum_{j=1}^{N-1} \eta^N_t(j)\,
\delta_{j/N}(dx) \;,
\]
where $\delta_{j/N}$ is the Dirac mass at $j/N$. 

Fix a measurable density profile $\rho_0: [0,1] \to [0,1]$ and let
$\{\mu_N : N \ge 1\}$ be a sequence of probability measures on
$\Sigma_N$ associated to $\rho_0$ in the sense that for every
continuous function $G:[0,1]\rightarrow{\mathbb{R}}$ and every
$\delta>0$,
\begin{equation*}
\lim_{N\rightarrow{+\infty}}\mu_N\Big(\Big|\frac{1}{N}
\sum_{k=1}^{N-1}G (k/N)\eta(k)-
\int_{0}^{1}G(x)\rho_{0}(x)\, dx\Big|>\delta\Big)=0\;.
\end{equation*}
Then, for each $t\ge 0$, $\pi_t^N$ converges in $\mathbb{
  P}_{\mu_N}$-probability to a measure which is absolutely continuous
with respect to the Lebesgue measure and whose density $\rho(t,x)$ is
the unique weak solution of the viscous Burgers equation with
Dirichlet's boundary conditions:
\begin{equation}
\label{vBe}
\left\{
\begin{array}{l}
\partial_t \rho = \partial_x^2\rho-E\,\partial_x b (\rho)\;, \\
\rho(t,0) = \alpha\;, \quad \rho(t,1) = \beta\;, \quad t\ge 0 \\
\rho(0,x) = \rho_0(x)\; ,\quad 0\leq x \leq 1\;,
\end{array}
\right.
\end{equation}
where $b(\rho) = \rho(1-\rho)$. We refer to \cite{g, DMPS, kl, blm1,
  flm1} and references therein.

\subsection{Nonequilibrium fluctuations.}
\label{neq}

To define the space in which the fluctuations take place, denote by
$C^2_0([0,1])$ the space of twice continuously differentiable
functions on $(0,1)$ which are continuous on $[0,1]$ and which vanish
at the boundary.  Let $- \Delta$ be the positive operator, essentially
self-adjoint on $L^2[0,1]$, defined by
\begin{equation*}
- \Delta\;=\; -\frac{d^2}{dx^2}\;, \quad
\mathcal{D}(-\Delta) \;=\; C^2_0([0,1]) \;. 
\end{equation*}
Its eigenvalues and corresponding (normalized) eigenfunctions have the
form $\lambda_n=(n\pi)^2$ and $e_n(x)=\sqrt{2}\sin(n\pi x)$
respectively, for any $n\ge 1$. By the Sturm-Liouville theory,
$\{e_n,\;n\ge 1\}$ forms an orthonormal basis of $L^2[0,1]$.

We denote with the same symbol the closure of $-\Delta$ in $L^2[0,1]$.
For any nonnegative integer $k$, we define the Hilbert spaces
$\mathcal{H}_k=\mathcal{D}(\{-\Delta\}^{k/2})$, with inner product
$(f,g)_k=(\{-\Delta\}^{k/2}f$, $\{-\Delta\}^{k/2}g)$, where $(\cdot
,\cdot )$ is the inner product in $L^2[0,1]$. By the spectral
theorem for self-adjoint operators,
\begin{equation*}
\mathcal{H}_k \;=\; \{f\in L^2[0,1]:\;\sum_{n=1}^{+\infty}n^{2k}
(f,e_n)^2<\infty\}\; ,
\end{equation*}
\begin{equation*}
(f,g)_k \;=\; \sum_{n=1}^{+\infty}(n\pi)^{2k}(f,e_n)(g,e_n)\;.
\end{equation*}

Moreover, if $\mathcal{H}_{-k}$ denotes the topological dual space of
$\mathcal{H}_k$,
\begin{eqnarray*}
\mathcal{H}_{-k} \;=\; \{f\in\mathcal{D}'(0,1):\;\sum_{n=1}^{+\infty}n^{-2k}
\<f,e_n\>^2<\infty\},\\ 
(f,g)_{-k} \;=\; \sum_{n=1}^{+\infty}(n\pi)^{-2k} \<f,e_n\> \<g,e_n\>,
\end{eqnarray*}
where $\mathcal{D}'(0,1)$ represents the space of distributions on
$(0,1)$ and $\<f,\cdot\>$ the action of the distribution $f$ on test
functions.

Fix a continuous density profile $\rho_0 :[0,1]\to [0,1]$, and denote
by $\rho(t,x)$ the unique weak solution of the viscous Burgers
equation \eqref{vBe}. Let $Y^N_t$ represent the density fluctuation
field which acts on functions $H$ in $C^1([0,1])$ as
\begin{equation*}
Y^N_t (H) \;=\; \frac 1{\sqrt N} \sum_{k=1}^{N-1} H(k/N)
\{ \eta_{t}(k) - \rho (t,k/N)\}\;. 
\end{equation*}

Fix $t>0$ and a function $G$ in $C^2_0([0,1])$. Recall that we denote
by $\rho(s,x)=\rho_s(x)$ the solution of the viscous Burgers equation
\eqref{vBe}. Let $(T_{t,s} G) (x) = G(s,x)$, $0\le s\le t$, be the
solution of the backward linear equation with final condition
\begin{equation} 
\label{f2.2}
\left\{
\begin{array}{l}
-\partial_s G  = \partial_x^2 G + E(1-2\rho_s) \partial_x G\;, \\
G(t, x)= G(x)\;, \quad  0\leq{x}\leq{1} \;, \\
G(s,0)=G(s,1)=0\;,\quad 0\leq{s}\leq{t}\;.
\end{array}
\right.
\end{equation}

Denote by $D([0,T], \mc H_{-k})$ the set of trajectories $Y:[0,T]\to
\mc H_{-k}$ which are right continuous and have left limits, endowed
with the Skorohod topology. 

\begin{theorem}
\label{s2.1}
Fix $T>0$, a positive integer $k> 7/2$, and a density profile $\rho_0
: [0,1]\to [0,1]$ in $C^4([0,1])$ such that $\rho_0(0)=\alpha$,
$\rho_0(1)=\beta$. Let $\{\mu_N : N\ge 1\}$ be a sequence of probability
measures on $\Sigma_N$ for which there exists a finite constant $A_2$
such that
\begin{equation}
\label{08}
\sup_{N\ge 1}\, \max_{1\le k\le N-1}  E_{\mu_N}
\Big[ \Big( \frac {1 }{\sqrt{N}} \sum_{j=1}^k
\big\{ \eta_0(j) - \rho_0(j/N) \big\} \Big)^4 \Big] \;\le\; A_2\;. 
\end{equation} 
Let $Q_N$ be the probability measure on $D([0,T], \mc H_{-k})$ induced
by the density fluctuation field $Y^N$ and the probability measure
$\mu_N$. Then, all limit points $Q^*$ of the sequence $Q_N$ are
concentrated on paths $Y$ such that for all $t\ge 0$ and $G$ in
$C^5_0([0,1])$,
\begin{equation*}
W(t,G)\;:=\;  Y_t (G) \;-\; Y_0(T_{t,0}G) 
\end{equation*}
are mean-zero Gaussian random variables with covariances given by
\begin{equation}
\label{f2.6}
E_{Q^*}[W(t,G) \, W(s,H)] \;=\; 
2 \int_{0}^{t\wedge s} \int_0^1 \sigma(\rho(r,x))\, (\partial_ x
T_{t,r} G)(x)\, (\partial_ x T_{s,r} H)(x) \, dx \, dr\;,
\end{equation}
for all $0\le s,t\le T$.  In this formula, $\sigma(\rho)$ represents
the mobility which is given by $\sigma(\rho) = \rho(1-\rho)$.
Moreover, for all $G$ and $H$ in $C^5_0([0,1])$, and $t>0$,
\begin{equation*}
E_{Q^*}[W(t,G) \, Y_0(H)] \;=\; 0\;.
\end{equation*}
\end{theorem}

\begin{corollary}
\label{s2.4}
In addition to the hypotheses of Theorem \ref{s2.1}, assume that
$Y^N_0$ converges to a zero-mean Gaussian field $Y$ with covariance
denoted by $\ll \cdot, \cdot \gg$, so that for all $G$, $H$ in
$C^2([0,1])$,
\begin{equation*}
\lim_{N\to\infty} E_{\mu^N} [Y^N_0(H) Y^N_0(G)] \;=\; \ll H, G \gg\; .
\end{equation*}
Then, the sequence $Q^N$ converges to a mean-zero Gaussian measure $Q$
whose covariances are given by
\begin{equation*}
\begin{split}
E_Q[Y_t(G) Y_s(H)] \; & =\; \ll T_{t,0} G , T_{s,0} H\gg \\
& +\; 2 \int_{0}^{t\wedge s} \int_0^1
\sigma(\rho(r,x))\, (\partial_ x T_{t,r} G)(x)\, 
(\partial_ x T_{s,r} H)(x) \, dx \, dr \;.    
\end{split}
\end{equation*}
for all $0\le s$, $t\le T$, $H$, $G$ in $C^5_0([0,1])$. 
\end{corollary}

This result is an immediate consequence of Theorem \ref{s2.1}. Under
any limit point $Q^*$ of the sequence $Q^N$, for any function $G$ in
$C^5_0([0,1])$, $Y_t (G)$ can be written as the sum of two
uncorrelated mean-zero Gaussian variables $W(t,G)$ and
$Y_0(T_{t,0}G)$.

Since under the measure $Q$, $W(t,G)$ is a Brownian motion changed in
time, the process $Y_t$ may be understood as a generalized
Ornstein-Uhlenbeck process described by the formal stochastic partial
differential equation
\begin{equation*}
dY_t \;=\; \ms L_t Y_t dt \;+\; \sqrt{2\sigma(\rho_t)} \nabla dW_t\;,
\end{equation*}
where $\ms L_t$ is the linear differential operator $\partial^2_x +
(1-2\rho_t) E \partial_x$.

The article is organized as follows. In Section \ref{sec3} we
introduce the microscopic Cole-Hopf transformation and we write the
density fluctuation field as the sum of a current field and a
remainder. In Section \ref{sec2.1} we prove Theorem \ref{s2.1} and
Corollary \ref{s2.4}, assuming that the density field $Y^N_t$ is tight
and that three estimates are in force. In Sections
\ref{sec5}--\ref{sec} we prove these three estimates, and in Section
\ref{sec2} we prove tightness of $Y^N_t$. All proofs rely on estimates
on the moments of the microscopic Cole-Hopf variables, presented in
Section \ref{sec4}, and on estimates of the solutions of certain
semi-discrete equations, presented in Section \ref{sec11}.

\section{A microscopic Cole-Hopf transformation}
\label{sec3}

To keep notation simple, from now on we drop the superscript $N$ on
the process $\eta_t^N$.  Following \cite{dg, g} we define in this
section a microscopic Cole-Hopf transformation of the process
$\eta_t$. For $N\ge 1$, let
\begin{equation*}
\Lambda^-_N \;=\; \{1, \dots, N-1\}\;, \quad
\Lambda_N \;=\; \{0, \dots, N-1\}\;, \quad
\Lambda^+_N \;=\; \{0, \dots, N\}\;.
\end{equation*}

Denote by $\Omega = \Omega^N$ the linear operator defined on
functions $f: \Lambda_N \to \bb R$ by
\begin{equation}
\label{fg11}
\left\{
\begin{array}{l}
\vphantom{\Big\{}
(\Omega f)(0) \;=\; - \alpha E N f(0) \;+\; N (\nabla_N^+ f) (0)\;, \\
\vphantom{\Big\{}
(\Omega f)(j) \;=\; (\Delta_N f)(j) \;-\; E (\nabla_N^- f)(j) \;,
\quad 1\le j\le N-2\;, \\
\vphantom{\Big\{}
(\Omega f)(N-1) \;=\; \beta E N f(N-1) \;-\; N \Big(1 + \frac{E}{N}\Big)
(\nabla_N^- f) (N-1)\;.
\end{array}
\right.
\end{equation}
In this formula,
\begin{equation*}
(\nabla_N^+ f)(j) = N[f(j+1) - f(j)]\;, \quad
(\nabla_N^-f)(j) = -N[f(j-1) - f(j)]\;,
\end{equation*}
and
\begin{equation*}
(\Delta_N f)(j) = N^2[f(j+1) + f(j-1) - 2f(j)].
\end{equation*}

Let $\lambda_t = \lambda^N_t$ be the solution of the linear equation
\begin{equation}
\label{fg9}
\left\{
\begin{array}{l}
\vphantom{\Big\{}
(\partial_t \lambda_t)(j)=(\Omega \lambda_t) (j)\;,
\quad 0 \leq j \leq N-1 ,\\
\vphantom{\Big\{}
\lambda_0(j) = \exp\big\{-(\gamma/N) \sum_{k=1}^j
\rho_0(k/N) \big\}\;,
\end{array}
\right.
\end{equation}
where $\gamma=\gamma_N \le 0$ is chosen so that $e^{-\gamma/N} = 1 +
E/N$, and $\rho_0:[0,1]\to[0,1]$ is a density profile satisfying the
assumptions of Theorem \ref{s2.1}. For $j\in\Lambda^-_N$, let
\begin{equation}
\label{r}
r_t(j) \;=\; -\frac{1}{\gamma}\, [\nabla^-_N
\ln (\lambda_t)] (j) \;.
\end{equation}

Denote by $\tilde Y^N_t$, $t\ge 0$, the modified density fluctuation
field defined on functions $G$ in $C^1([0,1])$ by
\begin{equation*}
\tilde{Y}^N_t (G) \;=\; \frac{1}{\sqrt{N}} \sum_{j=1}^{N-1}
G(j/N) \, \big\{ \eta_t(j)-r_t(j)\big \} \; .
\end{equation*}
Next result asserts that the original density fluctuation field
$Y^N_t$ is close to the modified density field $\tilde{Y}^N_t$.

\begin{proposition}
\label{s13}
For each $T>0$,
\begin{equation*}
\sup_{N\ge 1} \, \sup_{0\le t\le T} \, \max_{1\le j\le N-1}
N\, \big\vert r_t(j) - \rho (t, j/N) \big\vert \;<\; \infty \;.
\end{equation*}
\end{proposition}

For $0\le j,k\le N$ with $|j-k|=1$, denote by $J^{j,k}_t$, the total
number of jumps from $j$ to $k$ in the time interval $[0,t]$, and let
$W_t^{j,j+1}$ be the total current over the bond $\{j,j+1\}$, that is
\begin{equation*}
W_t^{j,j+1} \;=\; J_t^{j,j+1} \;-\; J_t^{j+1,j}\;.
\end{equation*}
In this formula, $J^{0,1}_t$ (resp. $J^{1,0}_t$) stands for the total
number of particles created (resp. removed) at the left boundary, with
a similar convention at the right boundary.

For $j\in\Lambda_N$, let
\begin{equation}
\label{fg1}
\xi_t(j) \;=\; \exp \big\{ (\gamma /N) \big[ W_t^{j,j+1} - \sum_{k=1}^j
\eta_0(k) \big] \big\}\;.
\end{equation}
Since
\begin{equation*}
\xi_t(j) \,-\, \xi_0(j) \,=\, \int_0^t \xi_{s-}(j) \, [e^{\gamma/N}-1]
\, dJ_s^{j,j+1} \,+\, \int_0^t \xi_{s-}(j)\, [e^{-\gamma/N}-1]\,
dJ_s^{j+1,j}\;,
\end{equation*}
$\xi_t(j)$ can be written as
\begin{equation}
\label{f2.11}
\xi_t(j) \;=\; \xi_0(j) \;+\; \int_0^t \xi_s(j) \, \mf g_{j,j+1}(\eta_s) \, ds
\;+\; \mathcal{M}^N_t(j)\;,
\end{equation}
where, in view of the definition of $\gamma$ and of the convention
\eqref{01}, 
\begin{equation*}
\mf g_{j,j+1}(\eta) \;=\; E N [\eta(j+1)-\eta(j)]\;,
\end{equation*}
and $\mathcal{M}^N_t(j)$ is a martingale with quadratic variation
given by
\begin{equation}
\label{VQuad}
\<\mathcal{M}^N (j),\mathcal{M}^N(k)\>_t \;=\;
\delta_{j,k} \, E^2 \int_0^t \xi_s(j)^2 \, \mf h_j(\eta_s) \, ds \;.
\end{equation}
In this formula, $\delta_{j,k}$ is the delta of Kroenecker and
\begin{equation}
\label{local rates g}
\mf h_j(\eta) \;:=\; e^{\gamma/N} \eta(j) \,[1-\eta(j+1)]
\;+\; \eta(j+1) \,[1-\eta(j)]\;.
\end{equation}

By the continuity equation, for $1\le j\le N-1$,
\begin{equation*}
W^{j-1,j}_t \;-\; W^{j,j+1}_t \;=\; \eta_t(j) \;-\; \eta_0(j)\;.
\end{equation*}
As a consequence, for $ 0\le j\le N-2$, $1\le k\le N-1$,
\begin{equation}
\label{fh1}
\begin{split}
&\xi_t(j+1) - \xi_t(j) \;=\; \xi_t(j)
\eta_t(j+1) \, [ \exp\{- \gamma/N\} -1] \;, \\
&\quad \xi_t(k-1) - \xi_t(k) \;=\;
\xi_t(k) \, \eta_t(k) \, [ \exp\{\gamma/N\} -1] \;.
\end{split}
\end{equation}
These equations explain the term $\sum_{1\le k\le j} \eta_0(k)$ in the
definition of $\xi_t(j)$. In view of the previous identities, by
definition of $\mf g_{j,j+1}$, and by the choice of $\gamma$,
\begin{equation}
\label{fg2}
\xi_t(j) \;=\; \xi_0(j) \;+\; \int_0^t (\Omega \xi_s)(j) \, ds
\;+\; \mathcal{M}^N_t(j)\;.
\end{equation}

The advantage of the process $\xi_t$ compared to the original process
$\eta_t$ is that it evolves according to the linear equation
\eqref{fg2}. Of course, the original process $\eta_t$ can be recovered
from $\xi_t$, since from \eqref{fg1} and by the continuity equation
appearing right below \eqref{local rates g}, for $1 \leq j \leq N-1$,
\begin{equation*}
\eta_t(j) \;=\; -\frac{1}{\gamma}\, [\nabla^-_N \ln (\xi_t)] (j)\;.
\end{equation*}

Denote by ${J}^N_t$, $t\ge 0$, the current fluctuation field
defined on functions $G\in C^1([0,1])$ by
\begin{equation*}
J^N_t (G) \;=\; \frac{1}{\sqrt{N}} \sum^{N-1}_{j=0}
\frac{(\nabla^+_NG)(j/N)}{\gamma \, \lambda_t(j)} 
\big( \xi_t(j) - \lambda_t(j) \big)  \; .
\end{equation*}
By the formula for $\eta_t(j)$ in terms of $\xi_t(j)$, and by
\eqref{r}, a summation by parts yields that for functions $G\in
C^1_0([0,1])$
\begin{equation}
\label{r1}
\tilde Y^N_t (G) \;=\; J^N_t (G) \;+\; R_t^N(G) \; ,
\end{equation}
where the remainder ${R}_t^N(G)$ is given by
\begin{equation*}
R_t^N(G) \;=\; \frac{1} {\sqrt{N}} \sum_{j=0}^{N-1}
\frac{1}{\gamma} (\nabla_N^+ G) (j/N)
\Big[ \ln \Big(\frac{\xi_t(j)}{\lambda_t(j)}\Big)
+1-\frac{\xi_t(j)}{\lambda_t(j)}\Big] \;.
\end{equation*}
Notice that both the current field ${J}^N_t $ and the remainder
${R}_t^N$ depend only on the process $\xi_t$.  Sometimes, by abuse of
notation, we consider that $R_t^N$ acts on discrete functions $g:\{0,
\dots, N\}\to \bb R$ instead of continuous functions $G:[0,1]\to\bb
R$. This is the case in the next proposition.

The second result of this section asserts that the modified density
fluctuation field $\tilde{Y}^N_t$ is close to the current
fluctuation field $J^N_t$.

\begin{proposition} 
\label{s10}
Fix $T>0$ and a function $\phi : [0,T]\times \Lambda^+_N \to \bb R$
such that
\begin{equation*}
\sup_{0\le t\le T} \, \max_{j\in\Lambda_N} \,
\frac{|(\nabla^+_N \phi_t)(j)|}{\lambda_{t}(j)} \;<\; \infty\;.
\end{equation*}
Then, for any $\delta >0$,
\begin{equation*}
\lim_{N\rightarrow{+\infty}} \mathbb{P}_{\mu_N} \Big[
\sup_{0\leq t\leq T} | R_t^N(\phi_t) | > \delta \Big] \;=\; 0\;.
\end{equation*}
\end{proposition}

\section{Proof of Theorem \ref{s2.1}}
\label{sec2.1}

Fix a density profile $\rho_0$ satisfying the assumptions of the
theorem and denote by $\rho(t,x)$ the solution of the viscous Burgers
equation \eqref{vBe} with initial condition $\rho_0$.  Let
$\{\mu_N:N\ge 1\}$ be a sequence of probability measures on $\Sigma_N$
for which \eqref{08} holds.

Let $\phi: \Lambda_N \to \bb R$ be a strictly positive
function. Denote by $\ms A_\phi = \ms A^N_\phi$ the difference
operator which acts on functions $g: \Lambda^+_N \to \bb R$ by
\begin{equation*} 
\left\{
\begin{split}
& (\ms A_\phi g)(0) \;=\;  (\ms A_\phi g)(N) =0 \; ,\\
& (\ms A_\phi g) (j) \;=\; 
(\Delta_N g)(j) \;+\; E\, \frac{[1-\theta_\phi (j)]}{1+ (E/N)
  \, \theta_\phi (j)} \, (\nabla_N^+g) (j)
\;-\; E \, \theta_\phi (j) \, (\nabla_N^-g) (j)
\end{split}
\right.
\end{equation*}
for $1\leq{j}\leq{N-1}$,  where
\begin{equation*}
\theta_\phi(j) \;=\; \frac{(\nabla_N^-\phi)(j)}
{E\, \phi(j-1)}\;\cdot
\end{equation*}

Denote by $\lambda_s$ the solution of \eqref{fg9}.  For $s\ge 0$, let
$\ms A_{s} = \ms A_{\lambda_s}$, and let
\begin{equation}
\label{f2.8}
\tilde{r}_s(j) \;:=\; \theta_{\lambda_s}(j) \;=\; \frac{(\nabla_N^-\lambda_s)(j)}
{E\, \lambda_s(j-1)}\;, \quad 1\le j\le N-1\;.
\end{equation}
By Lemma \ref{s12} below, $|\tilde{r}_t(j) - \rho(t,j/N)| \le C_0 /N$
uniformly in $0\leq{t}\leq{T}$ and $1\le j\le N-1$. Moreover, as $(\ms
A_s g)(0) = (\ms A_s g)(N) = 0$, the solution of the semi-discrete
equation
\begin{equation}
\label{f2.1}
\left\{
\begin{array}{l}
- (\partial_s g)(s,j) \;=\; (\ms A_s g)(s,j) \;, \quad
0\leq{j}\leq{N}\;, \\
g(t,j) \;=\; G(j/N) \;, \quad 0\leq{j}\leq{N}\;,
\end{array}
\right.
\end{equation}
for some $t>0$ and some $G$ in $C^2_0([0,1])$, is such that
$g_s(0)=g_s(N) = 0$ for all $0\le s\le t$. Hence, the semi-discrete
equation \eqref{f2.1} has to be understood as a discrete approximation
of the differential equation \eqref{f2.2}.

Fix a function $G$ in $C^{2}_0([0,1])$ and $t>0$.  Let $g_s(j) =
g^{N,t}_s (j)$ be the solution of \eqref{f2.1}.  A long computation
yields that for $0\le s\le t$,
\begin{equation}
\label{f2.5}
M^N_s(t,G) \;:=\;
J^N_s (g_s) \;-\; J^N_0 (g_0) \;=\; \frac{1}{\sqrt{N}} \sum_{j\in\Lambda_N} \int_0^s
\frac{(\nabla^+_N g_r)(j)}{\gamma \, \lambda_r(j)} 
\, d\mc M^N_r(j)\;,
\end{equation}
where $\mc M^N_s(j)$ is the martingale introduced in \eqref{f2.11}. We
present some details of this computation below equation \eqref{f2.7}.

\begin{proposition}
\label{s2.7}
Fix a density profile $\rho_0:[0,1]\to[0,1]$ and a sequence
$\{\mu_N:N\ge 1\}$ of probability measures on $\Sigma_N$ satisfying
the assumptions of Theorem \ref{s2.1}. Then, for each function $G$ in
$C^{2}_0([0,1])$ and $t>0$, there exists a finite constant $C_0$,
depending only on $G$ and $t$, such that for all $N\ge 1$,
\begin{equation*}
\bb E_{\mu_N} \big[ \sup_{0\le s\le t} M^N_s(t,G)^4
\big] \;\le \; C_0\;,\quad
\bb E_{\mu_N} \big[ \<M^N(t,G)\>_t^2 \big] \;\le \; C_0\;.
\end{equation*}
If $G$ belongs to $C^5_0([0,1])$, then the sequence of martingales
$M^N_s(t,G)$, $0\le s\le t$, converges in $D([0,t], \bb R)$ to a
mean-zero, continuous martingale, denoted by $M_s(t,G)$.  For $G_1$,
$G_2$ in $C^5_0([0,1])$, $t_1, t_2>0$, and $0\le s_j\le t_j$, the
covariances of $M_{s_1}(t_1,G_1)$ and $M_{s_2}(t_2,G_2)$ are given by
\begin{equation*}
\bb E [M_{s_1}(t_1,G_1) \, M_{s_2}(t_2,G_2)] \;=\; 
2 \int_{0}^{s_1\wedge s_2} \int_0^1 \sigma(\rho(r,x))\, (\partial_ x
T_{t_1,r} G_1)(x)\, (\partial_ x T_{t_2,r} G_2)(x) \, dx \, dr\;.
\end{equation*}
\end{proposition}

Since $M_s(t,G)$ is a continuous martingale whose quadratic variation is
deterministic, $M_s(t,G)$ is a Brownian motion changed in time. In
particular, $M_t(t,G)$ is a mean-zero Gaussian random variable. 

\begin{proof}[Proof of Theorem \ref{s2.1}]
Let $Q^*$ be a limit point of the sequence $Q_N$. Fix a function
$G\in{C_0^5([0,1])}$ and $t>0$. Let $g_s(j) =
g^{N,t}_s (j)$ be the solution of \eqref{f2.1} with final condition
equal to $G$. By \eqref{r1}, Proposition \ref{s13} and \eqref{f2.5}, 
\begin{equation*}
Y^N_t (G) \;-\; Y^N_0 (g_0) \;=\; M^N_t(t,G) \;+\; R^N_t(G) \;-\;
R^N_0(g_0)\;+\; \frac{C_N}{\sqrt{N}}\; ,
\end{equation*}
where $C_N$ is a sequence of numbers uniformly bounded.  By
Proposition \ref{s2.7} and in view of the remark made just after that
result, $M^N_t(t,G)$ converges in distribution to a mean-zero Gaussian
random variable, denoted by $W(t,G)$, whose variance is given by the
right hand side of \eqref{f2.6}, with $H=G$, $s=t$.

Let $\psi(s,j) = (\nabla_N^+ g_{t-s}) (j)/\lambda_{t-s} (j)$, $j\in
\Lambda_N$, $0\le s\le t$. By Remark \ref{s2.15} and by Proposition
\ref{s10}, $R^N_t(G)$ and  $R^N_0(g_0)$ converges to $0$ in probability.
Recall that we denote by $T_{t,s} G$ the solution of equation
\eqref{f2.2}. By Lemma \ref{s2.16}, $Y^N_0 (g_0) - Y^N_0(T_{t,0}G)$ is
absolutely bounded by $C_0/\sqrt{N}$. In conclusion, $Y^N_t (G) -
Y^N_0 (T_{t,0}G)$ converges in distribution to $W(t,G)$.

The covariance between $Y_0 (H)$ and $W_t (t,G)$ vanishes because $W_s
(t,G)$, $0\le s\le t$ is a martingale which vanishes at $s= 0$.

To complete the proof, it remains to compute the covariance between
$W(t,G)$ and $W(s,H)$. Assume that $s\le t$. Since $W_r (t,G)$, $0\le
r\le t$, is a martingale,
\begin{equation*}
E_{Q^*}[W(t,G) \, W(s,H)] \;=\; E_{Q^*}[W_s(t,G) \, W(s,H)] \;. 
\end{equation*}
By the polarization identity, we may express the covariance of a pair
of random variables $(X,Y)$ in terms of the variances of the variables
$X+Y$ and $X-Y$.
\end{proof}

\section{Proof of Proposition \ref{s13}}
\label{sec5}

The main result of this section asserts that the solution $\lambda_t$
of the linear equation \eqref{fg9} (satisfied by the expectation of
the Cole-Hopf variables $\xi_t$), is close to the Cole-Hopf
transformation of the solution of the viscous Burgers equation
\eqref{vBe}. 

Fix a profile $\rho_0:[0,1]\to [0,1]$ in $C^4([0,1])$, and denote by
$\rho(t,x)$ the solution of the hydrodynamic equation \eqref{vBe}. Let
$K(t,x)$ be the Cole-Hopf transformation of $\rho(t,x)$:
\begin{equation*}
K (t,x)\;=\;\exp\Big\{ E \Big[ \int_0^t \big\{ \partial_x\rho (s,x)
- E \, b(\rho(s,x))\big\} \,ds + \int_0^x \rho_0(y) \, dy \Big]
\Big\}\;.
\end{equation*}
Since $\partial_t K = K E [\partial_x\rho -E b (\rho)]$ and $\partial_x
K = E K \rho$, $K$ satisfies the linear parabolic equation with
boundary conditions
\begin{equation}
\label{eq_linearized}
\left\{
\begin{array}{l}
\p_t K = \partial_x^2 K  -E \partial_x K, \\
(\partial_x K)(t,0)= E \alpha K(t,0) \;,\quad
(\partial_x K)(t,1)= E \beta K(t,1)\;, \quad 0< t \leq T\;, \\
K(0,x)= \exp\{ E \int_0^x \rho_0(y)\, dy\} \;, \quad 0 \leq x \leq 1\;.
\end{array}
\right.
\end{equation}
As $\rho_0$ belongs to $C^4([0,1])$, $K_0$ belongs to $C^5([0,1])$,
and, by Lemma \ref{s2.13}, $K$ belongs to $C^{2,4}(\bb R_+\times
[0,1])$.

Denote by $\Vert f\Vert_M$ the sup norm of a function
$f:\Lambda_N, \Lambda^\pm_N\to\bb R$:
\begin{equation*}
\Vert f\Vert_M \;=\; \max_{j} \big\vert f (j)
\big\vert\;,
\end{equation*}
where the maximum is carried over the domain of definition of $f$. By
abuse of notation, if $G$ belongs to $C([0,1])$, $\Vert G\Vert_M$
represents $\max_{0\le j \le N} \vert G (j/N) \vert$.

\begin{lemma}
\label{s07}
Let $\lambda_t$ and $K_t$ be the solutions of \eqref{fg9} and
\eqref{eq_linearized}, respectively. Then, for every $T>0$,
\begin{equation*}
\begin{split}
&\sup_{N\ge 1} \,\,\sup_{0 \leq t \leq T}\,\,\max_{0\le j\le N-1} N\big\vert
\lambda_t(j) - K_t(j/N) \big\vert \;<\; +\infty \;, \\
& \quad \sup_{N\ge 1} \,\,\sup_{0 \leq t \leq T}\,\,\max_{1\le j\le N-1}
N\big\vert (\nabla^-_N \lambda_t) (j) - (\p_x K_t)(j/N) \big\vert
\;<\; +\infty\;.
\end{split}
\end{equation*}
\end{lemma}

\begin{proof}
Fix $T>0$. In this proof, $C_0$ represents a finite constant which may
depend on the parameters $E$, $\beta$, $\alpha$, on the initial
condition $\rho_0$, and on $T$.  Let $w_t (j):=\lambda_t(j) -
K_t(j/N)$.  A simple computation shows that
\begin{equation}
\label{eq0}
(\p_t w_t) (j) \;=\; (\Omega w_t)(j) \;+\; \varphi(t,j)\;,
\end{equation}
where $\Omega$ has been introduced in \eqref{fg11} and where $\varphi(t,j)$ is
given by
\begin{equation*}
\left\{
\begin{split}
& N \big\{ (\nabla_N^+K_t)(j/N) - \alpha \, E \, K_t(j/N)\big\} -
(\partial_t K_t)(j/N)\;,
\quad j=0 \;, \\
& [(\Delta_N-\partial_x^2)K_t] (j/N) \;-\;
E[(\nabla_N^--\partial_x)K_t](j/N)\;, \quad  1\leq{j}\leq{N-2}\;, \\
& E\beta N K_t(j/N) \;-\; (N+E) (\nabla_N^- K_t)(j/N)
- (\partial_t K_t)(j/N) \;, \quad j=N-1 \;.
\end{split}
\right.
\end{equation*}
In view of the boundary conditions satisfied by $K_t$, we may replace
in the previous equation $\alpha \, E \, K_t(0)$ by $(\partial_x
K_t)(0)$ and $E \, \beta \, K_t([N-1]/N)$ by $E \, \beta \,
\{K_t([N-1]/N)- K_t(1)\} + (\partial_x K_t)(1)$. After these
replacements, recalling that $K_t$ and $\rho_0$ belong to
$C^4([0,1])$, we obtain that $\varphi(t,j)$ is absolutely bounded by
$C_0N^{-1}$ for $j$ in $\{1, \dots, N-2\}$ and by $C_0$ for $j=0$ and
for $j=N-1$.

Let $G_t(j) = \varphi_t(j) \mb 1\{1\le j\le N-2\}$, $U_t(j) = \varphi_t(j) -
G_t(j)$ so that $|G_t(j)|\le C_0N^{-1}$.  We may represent the
solution $w_t$ of \eqref{eq0} as
\begin{equation*}
w_t \;=\; e^{\Omega t} w_0 \;+\; \int_0^t e^{\Omega (t-s)} (G_s +
U_s) \, ds\;.
\end{equation*}
By Lemma \ref{s05}, $\Vert e^{\Omega t} w_0\Vert_M$ is bounded by
$C_0 e^{C_0 t} \Vert w_0 \Vert_M \le C_0 N^{-1}$ and $\Vert e^{\Omega
  (t-s)} G_s\Vert_M$ is absolutely bounded by $C_0 e^{C_0 (t-s)}
N^{-1} \le C_0 N^{-1}$. Furthermore, since $U_s$ vanishes everywhere
except at two points, by Corollary \ref{s03}, $\Vert e^{\Omega (t-s)}
U_s\Vert_M \le C_0 (t-s)^{-1/2} N^{-1}$ for all $N$ large
enough. Putting together all the previous estimates, we conclude that
$\Vert w_t\Vert_M$ is bounded by $C_0 N^{-1}$, proving the first
assertion of the lemma.

We turn to the second assertion. Let
\begin{equation*}
\gamma_t(j)=\left\{
\begin{split}
& [N/(N+E)] \alpha E \lambda_t(0)\;, \quad j=0\\
& (\nabla^-_N \lambda_t)(j)
\;, \quad 1\le j\le N-1\;, \\
&   \beta E \lambda_t(N-1)\;, \quad j=N\;.
\end{split}
\right.
\end{equation*}
It is not difficult to show that for $1\le j\le N-1$, $\gamma_t$
solves the equation
\begin{equation*}
\p_t \gamma_t(j) = (\Delta_N \gamma_t)(j) - E (\nabla_N^-\gamma_t)(j)\;. 
\end{equation*}
Clearly, $(\partial_x K)$ satisfies a similar equation where the
discrete differential operators are replaced by continuous
ones. Therefore, in view of \eqref{eq_linearized}, $w_t(j) =
\{\gamma_t(j) - (\partial_x K)(t,j/N)\}$, $0\le j \le N-1$, satisfies
\begin{equation}
\label{f02}
\left\{
\begin{split}
&w_t(0) =  \alpha E \{ [N/(N+E)] \lambda_t(0) - K(t,0)\}\;,\\
& \p_t w_t(j) = (\Delta_N w_t)(j) - E (\nabla_N^-w_t)(j) + \varphi(t,j)
\;, \quad 1\le j\le N-1\;, \\
& w_t(N) = \beta E \{ \lambda_t(N-1) - K(t,1)\}\;,
\end{split}
\right.
\end{equation}
where $\varphi(t,j)$ accounts for the difference between the discrete and
continuous derivatives, namely
\begin{equation*}
\varphi(t,j) \;=\; (\Delta_N v_t)(j/N) \;-\; (\partial^2_x v)(t,j/N)
\;-\; E \big\{ (\nabla_N^- v_t)(j/N) - (\partial_x v)(t,j/N)
\big\}\;,
\end{equation*}
where $v(t,j) = (\partial_x K)(t,j/N)$. 

Since $K_t$ belongs to $C^4([0,1])$, $\varphi$ is absolutely bounded by $C_0
N^{-1}$ uniformly in $t$ and $j$. By the first part of the proof and
by Lemma \ref{s05}, $w_t(0)$ and $w_t(N)$ are also absolutely bounded
by $C_0 N^{-1}$.

Let $w^*_t(j)$ be the solution of \eqref{f02} with the same initial
condition satisfied by $w_t(j)$, but with boundary conditions
$w^*_t(0) = C/N$, $w^*_t(N) = C/N$, where $C$ is a finite constant
such that $w_t(0) \vee w_t(N) \le C/N$ for all $0\le t\le T$. By the
maximum principle, $w_t(j) \le w^*_t(j)$ for $0\le t\le T$, $0\le j\le
N$. Denote by $\Omega_\dagger$ the generator of a weakly asymmetric random
walk on $\{0, \dots, N\}$ absorbed at $0$ and $N$. We may represent
$w^*_t$ as
\begin{equation*}
w^*_t \;=\; e^{\Omega_\dagger t} w_0 \;+\; \int_0^t e^{\Omega_\dagger (t-s)} \varphi_s
\, ds\;, 
\end{equation*}
and repeat the arguments presented in the first part of the proof to
conclude that $\Vert w^*_t\Vert_M \le C_0/N$. This provides an upper
bound for $w_t$. A lower bound can be derived along the same lines.
\end{proof}

Recall the definition of $\tilde{r}_t$ given in \eqref{f2.8}.

\begin{lemma}
\label{s12}
For every $T>0$,
\begin{equation*}
\sup_{N\geq 1} \sup_{0\le t\le T} \max_{1\le j\le N-1}
N \, \big\vert \tilde{r}_t(j) - \rho (t, j/N)  \big\vert \;<\; \infty.
\end{equation*}
\end{lemma}

\begin{proof}
By definition of $\tilde{r}_t$ and by the uniform lower bound for
$\lambda_{t}$, proved in Lemma \ref{s08},
\begin{equation*}
\big\vert \tilde{r}_t(j) - \rho (t, j/N)  \big\vert \;\leq\;
C_0\, \big| (\nabla_N^-\lambda_t)(j)- E\lambda_t(j-1)\rho (t, j/N) \big|
\end{equation*}
for some finite constant $C_0$, whose value may change from line to
line.  Since $(\partial_x K_t)(j/N) = E\,\rho (t,j/N)\,K_t(j/N)$ and
since $\rho$ is bounded, the right hand side of the previous
expression is less than or equal to
\begin{equation*}
C_0 \Big\{ \big |(\nabla_N^-\lambda_t)(j) - (\partial_x K_t)(j/N)
\big| \;+\; \big| K_t(j/N)-\lambda_t(j-1)\big| \Big\}\;.
\end{equation*}
The result follows from Lemma \ref{s07} and the smoothness of $K$.
\end{proof}

\begin{lemma}
\label{s2.10}
For every $T>0$,
\begin{equation*}
\sup_{N\geq 1} \,\, \sup_{0\leq t\leq T} \,\, \max_{1\leq j\leq N-2}
\big|\nabla_N^+\tilde{r}_t(j)\big| \;<\; \infty\;.
\end{equation*}
\end{lemma}

\begin{proof}
Write
\begin{equation*}
\begin{split}
\big| \nabla_N^+\tilde{r}_t(j) \big| \; &\leq\; 
N\big|\tilde{r}_t(j+1)-\rho (t, [j+1]/N )\big| 
\;+\; N\big|\rho (t,[j+1]/N)-\rho(t, j/N)\big| \\
& +\; N\big|\rho (t, j/N)-\tilde{r}_t(j)\big|\;.
\end{split}
\end{equation*}
The first and third terms on the right hand side of the last
expression are bounded by the previous lemma. To complete the proof it
remains to recall that $\rho$ is of class $C^{1,2}$.
\end{proof}

\begin{proof}[Proof of Proposition \ref{s13}]
By Lemma \ref{s12}, it is enough to show that
\begin{equation}
\label{f03}
\sup_{0\le t\le T} \, \, \max_{1\le j\le N-1}
N\, \big\vert r_t(j) - \tilde{r}_t(j) \big\vert \;\leq\; C_0\;.
\end{equation}

By definition of $r_t$ and $\gamma$, for $1\leq j\leq N-1$
\begin{equation*}
r_t(j) \;=\; \frac{\log \big ( 1+ [E/N] \, \tilde{r}_t(j) \big)}
{\log (1+ [E/N])} \;\cdot
\end{equation*}
Since, by Lemma \ref{s02}, 
\begin{equation*}
0 \;\le\;  \tilde{r}_t(j) \;\le\;  1\;,
\end{equation*}
for $1\leq j\leq N-1$, $0\le t\le T$, \eqref{f03} holds, which
completes the proof of the proposition.
\end{proof}

\section{Proof of Proposition \ref{s10}}
\label{sec6}

Fix $T>0$ and a sequence of probability measures $\{\mu_N: N\ge 1\}$
fulfilling \eqref{08}.

\begin{lemma}
\label{s11}
For every $T>0$ and $\delta>0$,
\begin{equation*}
\lim_{N\rightarrow{\infty}} \mathbb{P}_{\mu_N}
\Big[ \sup_{0\leq t\leq T} \frac{1}{\sqrt{N}}
\sum_{j\in\Lambda_N}[\xi_{t}(j)-\lambda_{t}(j)]^2
\;> \; \delta \Big] \;=\; 0\;.
\end{equation*}
\end{lemma}

\begin{proof}
Fix $T>0$ and $\tau>0$. It is enough to show that for an appropriate
choice of $\tau$, for each $\delta>0$
\begin{equation}
\label{07}
\lim_{N\rightarrow{\infty}}\sup_{0\leq{t}\leq{T}}\frac{1}{\tau}
\, \mathbb{P}_{\mu_N} \Big[
\sup_{t\leq s\leq t+\tau} \frac{1}{\sqrt{N}}
\sum_{j\in\Lambda_N}[\xi_{s}(j)-\lambda_{s}(j)]^2
 >\delta \Big] \;=\; 0\;.
\end{equation}

A long and simple computation shows that for $t\le s$,
\begin{equation}
\label{06}
\begin{split}
& \frac{1}{\sqrt{N}} \sum_{j=0}^{N-1}[\xi_{s}(j)-\lambda_{s}(j)]^2
\;-\; \frac{1}{\sqrt{N}} \sum_{j=0}^{N-1}[\xi_{t}(j)-\lambda_{t}(j)]^2
\\
& \quad =\; \int_t^s \frac{2}{\sqrt{N}} \sum_{j=0}^{N-1}
(\xi_{r} -\lambda_{r})(j)\, [\Omega (\xi_{r} - \lambda_r)](j)\, dr
\\
& \qquad +\; \int_t^s \frac{1}{\sqrt{N}} \sum_{j=0}^{N-1}
\Big\{ (\Omega_2 \xi^2_{r})(j) - 2 \xi_{r}(j) (\Omega
\xi_{r})(j) \Big\} \, dr
\\
& \qquad -\; \int_t^s  \frac{a_N}{\sqrt{N}} \sum_{j=0}^{N-1}
\xi^2_{r}(j) \eta_{r}(j) \eta_{r}(j+1)  \, dr
\;+\; \big\{M_s - M_t\big\}\;,
\end{split}
\end{equation}
where $a_N= N^2\{e^{\gamma/N} - e^{-\gamma/N} + e^{-2\gamma/N} -1\}$
is a positive constant and $M_t$ a martingale.

Consider a sequence $\tau=\tau_N$ such that $N^{-1} \ll \tau_N \ll
N^{-2/3}$. We show below that with this choice \eqref{07} holds for
each term of the previous decomposition. For instance, by Lemma
\ref{sd4} and Tchebycheff inequality,
\begin{equation*}
\lim_{N\rightarrow{\infty}}\sup_{0\leq{t}\leq{T}}\frac{1}{\tau}
\, \mathbb{P}_{\mu_N} \Big[
\frac{1}{\sqrt{N}} \sum_{j=0}^{N-1}[\xi_{t}(j)-\lambda_{t}(j)]^2
>\delta \Big] \;=\; 0\;.
\end{equation*}
Hence, \eqref{07} holds for the second term on the left hand side of
\eqref{06} provided $N^{-1} \ll \tau_N$.

Repeating the arguments presented in the proof of Lemma \ref{s06}, we
can show that the expression inside the first integral on the right
hand side of \eqref{06} is bounded by
\begin{equation*}
\frac{C_0}{\sqrt{N}} \sum_{j=0}^{N-1} [\xi_{r}(j) -\lambda_{r}(j)]^2
\end{equation*}
for some finite constant $C_0$. To show that \eqref{07} holds for this
term it is therefore enough to apply Markov inequality and to
recall the statement of Lemma \ref{sd4}. No condition on $\tau_N$ is
needed in this argument due to the time integral.

% para este termo, so precisamos de uma estimativa para o quadrado no
% Lema sd4  

The expression inside the integral in the second term on the right
hand side of \eqref{06} is bounded by
\begin{equation*}
\frac{C_0}{\sqrt{N}} \Big\{ \sum_{j=0}^{N-1} N^2[\xi_{r}(j+1) - \xi_{r}(j)]^2 \;+\;
N \xi_r(0)^2 \;+\; N \xi_r(N-1)^2 \Big\}
\end{equation*}
for some finite constant $C_0$. By \eqref{02}, $\xi_r(0)^2$ and
$\xi_r(N-1)^2$ are bounded above by $C_0 N^{-1} \sum_{j\in\Lambda_N}
\xi_{r}(j)^2$, and $|\xi_{r}(j+1) - \xi_{r}(j)|$ is less than or equal
to $ (e^{-\gamma/N}-1) \xi_{r}(j)$. The previous expression is thus
less than or equal to $C_0 N^{-1/2} \sum_{j\in\Lambda_N}
\xi_{r}(j)^2$.  By Tchebycheff and H\"older inequalities,
\begin{equation*}
\mathbb{P}_{\mu_N}
\Big[  \sup_{t\leq s\leq t+\tau} \int_t^s
\frac{1}{\sqrt{N}} \sum_{j=0}^{N-1} \xi_{r}(j)^2\, dr
\;> \; \delta \Big] \;\le\; \frac {N^2 \tau^3}{\delta^4}
\mathbb{E}_{\mu_N}  \Big[ \int_t^{t+\tau}
\frac{1}{N} \sum_{j=0}^{N-1} \xi_{r}(j)^8\, dr \Big]\;.
\end{equation*}
By Lemma \ref{s09}, this expression is bounded above by $C_0 N^2
\tau^4 \delta^{-4}$. The contribution of the second term on the right
hand side of \eqref{06} to \eqref{07} is thus bounded by $C_0 N^2
\tau^3 \delta^{-4}$, which vanishes, as $N\to\infty$, provided $\tau_N
\ll N^{-2/3}$.

Since the third term in \eqref{06} is negative, it remains to consider
the martingale $M_t$. Its quadratic variation $\<M\>_t$ is such that
\begin{equation*}
\<M\>_s \;-\; \<M\>_t \;\le\;  \int_t^s
\frac{C_0}N \sum_{j=0}^{N-1} \xi_r(j)^2 \Big\{ \frac 1{N^2} \, \xi_r(j)^2 +
[\xi_r(j) - \lambda_r(j)]^2 \Big\}\, dr
\end{equation*}
for some finite constant $C_0$ and all $t\le s$. Therefore, by Doob's
inequality,
\begin{equation*}
\begin{split}
& \mathbb{P}_{\mu_N} \Big[
\sup_{t\leq s\leq t+\tau} |M_s - M_t| >\delta \Big] \\
&\quad \le\; \frac{C_0}{\delta^2} \, \mathbb{E}_{\mu_N} \Big[
\int_t^{t+\tau} \frac{1}N \sum_{j=0}^{N-1} \xi_r(j)^2 \Big\{
\frac 1{N^2} \, \xi_r(j)^2 + [\xi_r(j) - \lambda_r(j)]^2 \Big\}\, dr
\Big]\;.
\end{split}
\end{equation*}
By Lemmas \ref{s09} and \ref{sd4}, this expectation is bounded above
by $C_0 \tau N^{-1}$. Hence, \eqref{07} holds for the martingale part
in \eqref{06}, which proves the lemma.
\end{proof}

% para a primeira afirmacao do corolario preciso de muito menos

\begin{corollary}
\label{s2.9}
For every $T>0$, $\delta>0$ and $a<1$,
\begin{equation*}
\lim_{N\to\infty} \mathbb{P}_{\mu_N} \Big[
\sup_{0\leq{t}\leq{T}} |\xi_{t}(0)-\lambda_{t}(0)| > \delta \Big]\;=\;
0\;, \quad 
\lim_{N\rightarrow \infty} \mathbb{P}_{\mu_N}
\Big[\inf_{0\leq{t}\leq{T}} \frac{\xi_{t}(0)}{\lambda_{t}(0)}
<a\Big]\;=\;0 \;.   
\end{equation*}
\end{corollary}

\begin{proof}
By the triangular inequality, by Lemma \ref{s02} and by
\eqref{02}, $[\xi_{t}(0)-\lambda_{t}(0)]^2$ is bounded by
\begin{equation*}
C_1 \Big\{ \Big(\frac jN\Big)^2 \xi_t(0)^2 \;+\;
[\xi_t(j) - \lambda_t(j)]^2 \;+\; \Big(\frac jN\Big)^2 \lambda_t(0)^2
\Big\}\;,
\end{equation*}
for some finite constant $C_1$ and all $j\in\Lambda_N$. In view of
Lemma \ref{s09} and Lemma \ref{s05}, averaging over $0\le j\le
\epsilon N$, the first assertion of the corollary follows from Lemma
\ref{s11}.

By Lemma \ref{s08}, there exists a positive constant $c_0$, depending
only on $\rho_0$, $E$, $\alpha$, $\beta$ and $T$, such that
$\lambda_t(j) \ge c_0$ for all $0\le t\le T$, $0\le j\le N-1$.  Let
$\delta = c_0(1-a)>0$ so that
\begin{equation*}
\mathbb{P}_{\mu_N} \Big[\inf_{0\leq{t}\leq{T}}
\frac{\xi_{t}(0)}{\lambda_{t}(0)} <a \Big]
\;\le\; \mathbb{P}_{\mu_N} \Big[
\sup_{0\leq{t}\leq{T}} |\xi_{t}(0)-\lambda_{t}(0)| > \delta \Big]\;.
\end{equation*}
Hence, the second assertion of the corollary follows from the first one. 
\end{proof}

% nesta prova nao preciso do sup em t.

\begin{proof}[Proof of Proposition \ref{s10}]
By Lemma \ref{s02} and by \eqref{02}, $\xi_t(j)/\lambda_t(j) \ge
e^{\gamma} \xi_t(0)/\lambda_t(0)$ for all $j\in\Lambda_N$. Therefore,
by the second assertion of Corollary \ref{s2.9}, for every
$a<e^\gamma$,
\begin{equation*}
\lim_{N\rightarrow \infty} \mathbb{P}_{\mu_N}
\Big[\inf_{0\leq{t}\leq{T}} \, \min_{0\le j\le N-1}
\frac{\xi_{t}(j)}{\lambda_{t}(j)} < a \Big]\;=\;0 \;.
\end{equation*}
Fix $a<e^\gamma$ and denote by $\Lambda^c_a$ the previous set of
trajectories.

For each $0<\delta <1$ there exists a finite constant $C(\delta)$ such
that
\begin{equation*}
|\log (z) + 1 - z| \leq C(\delta)\, |1-z|^2 \;,\quad z\ge \delta\;.
\end{equation*}
Therefore, on the set $\Lambda_a$, by Lemma \ref{s08} applied to the
function $\lambda_t$, for every function $\phi : [0,T]\times
\Lambda^+_N \to \bb R$ satisfying the assumptions of the proposition,
\begin{equation*}
|R_t^N(\phi_t)| \;\leq\; \frac{C_1} {\sqrt{N}} \sum_{j=0}^{N-1} 
|(\nabla^+_N \phi_t)(j)|\,
\frac{(\xi_t(j)-\lambda_t(j))^2}{\lambda_{t}^2(j)}
\;\leq\; \frac{C'_1} {\sqrt{N}} \sum_{j=0}^{N-1}
(\xi_t(j)-\lambda_t(j))^2\;,
\end{equation*}
for some finite constant $C_1$. Hence, the assertion of the
proposition follows from Lemma \ref{s11}.
\end{proof}

\section{Proof of Proposition \ref{s2.7}}
\label{sec}

Fix a density profile $\rho_0$ satisfying the assumptions of Theorem
\ref{s2.1} and denote by $\rho(t,x)$ the solution of the viscous
Burgers equation \eqref{vBe} with initial condition $\rho_0$.  Let
$\{\mu_N:N\ge 1\}$ be a sequence of probability measures on $\Sigma_N$
for which \eqref{08} holds.

Denote by $\Omega^*$ the adjoint operator of $\Omega$ with respect to
the counting measure. An elementary computation gives that
\begin{equation*}
\left\{
\begin{array}{l}
\vphantom{\Big\{}
(\Omega^* f)(0) \;=\; (1 - \alpha) E N f(0) \;+\; (E+N) (\nabla_N^+
f)(0) \;, \\
\vphantom{\Big\{}
(\Omega^* f)(j) \;=\; (\Delta_N f)(j) \;+\; E (\nabla_N^+ f)(j) \;,
\quad 1\le j\le N-2\;, \\
\vphantom{\Big\{}
(\Omega^* f)(N-1) \;=\; - (1-\beta) E N f(N-1) \;-\; N 
(\nabla_N^- f) (N-1)\;.
\end{array}
\right.
\end{equation*}
Note that $\Omega^*$ has exactly the same structure as $\Omega$.  Fix
a function $\psi : \Lambda_N \to \bb R$, and denote by $\psi(s,j)$,
$j\in \Lambda_N$, $s\ge 0$ the solution of
\begin{equation}
\label{f2.10}
\left\{
\begin{split}
& \partial_s \psi_s = \Omega^* \psi_s \;,\\  
& \psi_0(j) = \psi (j)\;.
\end{split}
\right.
\end{equation}

\begin{lemma}
\label{s2.14}
Assume that $F$ belongs to $C^4([0,1])$ and let $F(t,x)$ be the
solution of the linear equation
\begin{equation*}
\left\{
\begin{split}
& \partial_s F \;=\; \partial^2_x F  \;+\; E \partial_x F\;, \\  
& F(0,x) \;=\; F(x)\;,\; x\in [0,1]\;,
\end{split}
\right.
\end{equation*}
with boundary conditions 
\begin{equation*}
(\partial_x F) (s,0) = - (1-\alpha) E F(s,0)\;, \quad 
(\partial_x F) (s,1) = - (1-\beta) E F(s,1)\;, \quad s\ge 0 \;.
\end{equation*}
Suppose that there exists a finite constant $C_0$ such that 
\begin{equation*}
\max_{j\in\Lambda_N} \big|\psi(j) - F(j/N)  \big\vert \;\le\;
C_0/N\; \cdot
\end{equation*}
Then, for every $T>0$, there exists a finite constant $C_0$ such that 
\begin{equation*}
\sup_{0 \leq t \leq T} \max_{j\in\Lambda_N} \big\vert \psi_t (j) 
- F_t (j/N) \big\vert \;\le\; C_0/N\;. 
\end{equation*}
\end{lemma}

\begin{proof}
By the note following Lemma \ref{s2.13}, $F$ belongs to $C^{1,3}(\bb
R_+\times [0,1])$.  As in the proof of Lemma \ref{s07}, let $w_t
(j):=\psi_t(j) - F(t,j/N)$.  As $F$ belongs to $C^{1,3}(\bb R_+\times
[0,1])$, equation \eqref{eq0} holds with $\Omega$ replaced by
$\Omega^*$ for some function $\varphi(t,j)$ which is absolutely
bounded by $C_0N^{-1}$ for $j$ in $\{1, \dots, N-2\}$ and by $C_0$ for
$j=0$ and for $j=N-1$. Since, by assumption, the initial condition $w_0$
is also uniformly bounded by $C_0/N$, the arguments presented in the
proof of the first assertion of Lemma \ref{s07} yield that $\psi_t(j)
- F_t(j/N)$ is uniformly bounded by $C_0/N$.
\end{proof}

Recall the definition of the operator $\ms A_\phi$ introduced at the
beginning of Section \ref{sec2.1}. The proof of Proposition \ref{s2.7}
relies on the following remarkable identity, derived from a long, but
elementary, computation. For every pair of functions $g : \Lambda^+_N
\to \bb R$, $\phi : \Lambda_N \to \bb R$,
\begin{equation}
\label{f2.7}
\Omega^* \Big( \frac {\nabla^+ g} {\phi} \Big) (j) \;-\; 
(\nabla^+ g)(j)\, \frac{(\Omega  \phi)(j)}{\phi(j)^2} \;=\; 
\frac {[\nabla^+ (\ms A_\phi \, g)](j)}{\phi(j)} \;, \quad j\in
\Lambda_N\;.
\end{equation}

Identity \eqref{f2.7} explains the second identity in
\eqref{f2.5}. Indeed, for a time-independent function $g: \Lambda^+_N
\to \bb R$, since $\partial_s \lambda^{-1}_s = - \lambda_s^{-2} \,
\Omega \lambda_s$, due to \eqref{fg9}, \eqref{fg2} and an integration
by parts,
\begin{equation}
\label{f2.14}
\begin{split}
& J^N_s (g) \;-\; J^N_0 (g) \; =\; \frac{1}{\sqrt{N}} \sum_{j\in\Lambda_N} \int_0^s
\frac{(\nabla^+ g)(j)}{\gamma \, \lambda_r(j)} \, d\mc M^N_r(j) \\
& \quad\qquad +\; \frac{1}{\gamma \sqrt{N}} \sum_{j\in\Lambda_N} \int_0^s
\Big\{ \Omega^* \Big( \frac {\nabla^+ g} {\lambda_s} \Big) (j)
\;-\; (\nabla^+ g)(j)\, \frac{(\Omega  \lambda_s)(j)}{\lambda_s(j)^2}
\Big\}\, \xi_r(j) \, dr \;.
\end{split}
\end{equation}
By \eqref{f2.7}, the expression inside braces in the previous equation
is equal to $[\nabla^+ (\ms A_s \, g)](j)/\lambda_s (j)$, where $\ms
A_s = \ms A_{\lambda_s}$.  Hence, if we consider a time-dependent
function $g_s$ which solves \eqref{f2.1}, the additive part in the
previous decomposition of $J^N_s (g_s) \;-\; J^N_0 (g_0)$ vanishes,
yielding \eqref {f2.5}. \smallskip

\begin{remark}
\label{s2.15}
Fix a function $G$ in $C^2_0([0,1])$ and $t>0$. Let $g_s$ be the
solution of \eqref{f2.1} with final condition equal to $G$, $g(t,j) =
G(j/N)$, and let $\psi(s,j) = (\nabla_N^+ g_{t-s}) (j)/\lambda_{t-s}
(j)$, $j\in \Lambda_N$, $0\le s\le t$. By \eqref{f2.1} and
\eqref{f2.7}, in the time interval $[0,t]$, $\psi(s,j)$ solves the
equation \eqref{f2.10} with initial condition
\begin{equation*}
\psi_0(j) = (\nabla_N^+ G) (j/N)/\lambda_{t}(j)\;.
\end{equation*}
In particular, by Lemmas \ref{s08} and \ref{s05}, there exists a
finite constant $C_0$ such that for all $N\ge 1$,
\begin{equation}
\label{f2.19}
\sup_{0 \leq s\leq t}\,\,\big\Vert \psi_s \big\Vert_M \;\le\; C_0 \;.
\end{equation}
\end{remark} 

\begin{remark}
\label{s2.17}
Similarly, let $G(s,x)$ be the solution of \eqref{f2.2} with final
condition $G(t,x)=G(x)$. A computation, based on a continuous version
of equation \eqref{f2.7}, shows that in the time interval $[0,t]$, the
function $F_s = \partial_x G_{t-s}/K_{t-s}$ solves the equation
appearing in the statement of Lemma \ref{s2.14} with initial condition
$F(0,x)= (\partial_x G)(x)/K(t,x)$.
\end{remark}

Therefore, if $G$ belongs to $C^5_0([0,1])$, since $K$ belongs to
$C^{2,4}(\bb R_+\times [0,1])$, $F(0,x) = (\partial_x
G)(x)/K_{t}(x)$ belongs to $C^4([0,1])$. Moreover, by Lemmas \ref{s08}
and \ref{s07}, $\psi(0,j) - F(0,j/N)$ is uniformly bounded by
$C_0/N$. Therefore, by Lemma \ref{s2.14}, there exists a finite
constant $C_0$ for which for all $N\ge 1$,
\begin{equation}
\label{f2.20}
\sup_{0 \leq s\leq t}\, \max_{j\in\Lambda_N} \,\big\vert \psi_s(j) - F_s(j/N)
\big\vert \;\le\; C_0/N\;. 
\end{equation}

\begin{lemma}
\label{s2.16}
Fix $G$ in $C^5_0([0,1])$ and $t>0$. Denote by $G(s,x)$ the solution of
\eqref{f2.2} with final condition equal to $G$, and by $g$ the solution
of \eqref{f2.1} with the same final condition. Then, there exists a
finite constant $C_0$ such that for all $N\ge 1$,
\begin{equation*}
\big\Vert G(0,\cdot) - g(0,\cdot) \big\Vert_M \;\le\; C_0/N\;.  
\end{equation*}
\end{lemma}

\begin{proof}
Since $G(s,0)=g_s(0)=0$ for $0\le s\le t$, for every $j\in\Lambda_N$,
by Remarks \ref{s2.15} and \ref{s2.17},
\begin{equation*}
\begin{split}
\big\vert G(0,j/N) - g_0(j) \big\vert \; &\le\;
\frac 1N \sum_{k=0}^{j-1} \big\vert  (\nabla^+_N G)(0,k/N) -  
(\nabla^+_N g_0)(k/N) \big\vert \\
\; & =\; \frac 1N \sum_{k=0}^{j-1} \Big\vert  N \int_{k/N}^{(k+1)/N}
F(t,y) K(0,y) \, dy -  \psi_t(k/N) \lambda_0(k) \Big\vert\;.
\end{split}
\end{equation*}
We have seen just before the statement of the lemma, that under the
assumptions that $G$ belongs to $C^5_0([0,1])$, $F(0,\,\cdot\,)$
belongs to $C^4([0,1])$. Therefore, by the proof of Lemma
\ref{s2.14}, $F$ belongs to $C^{1,3}([0,t]\times [0,1])$. The
assertion of the lemma follows from this remark, from the fact that
$\rho_0$ belongs to $C^1([0,1])$ and from \eqref{f2.20}.
\end{proof}

\begin{lemma}
\label{s2.5}
For each function $G$ in $C^{5}_0([0,1])$ and $t>0$, the quadratic
variation $\<M^N(t,G)\>_s$ of the martingale $M^N_s(t,G)$ converges in
$L^1(\bb P_{\mu_N})$ to
\begin{equation*}
2 \int_{0}^{s} \int_0^1 \sigma(\rho(r,x))\, [(\partial_ x
T_{t,r} G)(x)]^2 \, dx \, dr\;,
\end{equation*}
where $T_{t,r} G$ is the solution of \eqref{f2.2}.
\end{lemma}

\begin{proof}
With the notation introduced just before the statement of the lemma, the
quadratic variation of the martingale $M^N_s(t,G)$ can be written as 
\begin{equation}
\label{f2.9}
\<M^N(t,G)\>_s \;=\; \int_{0}^{s} \frac{E^2}{\gamma^2N} \sum_{j\in\Lambda_{N}}\,
\xi_r(j)^2 \, \mf h_j (\eta_r) \, \psi_{t-r}(j)^2 \, dr \;.
\end{equation}
By \eqref{f2.19}, $\psi$ is uniformly bounded in the time interval
$[0,t]$. Since the cylinder functions $\mf h_j$ are also bounded, by Lemma
\ref{sd4}, we may replace $\xi_r(j)^2$ by $\lambda_r(j)^2$ in the
previous formula paying the price of an error which converges to $0$
in $L^1(\bb P_{\mu_N})$.

For two functions $f$, $g:\Lambda_N \to \bb R$, and $1\le \ell\le
N/2$, since $b^2-a^2 = (b-a)(b+a)$,
\begin{equation*}
\frac 1N \sum_{j=\ell}^{N-1-\ell} \frac 1{2\ell + 1} \sum_{k=-\ell}^{\ell} [f(j+k)^2 -
f(j)^2] \, g(j)
\;\le\; \frac{4 \ell \Vert f \Vert_M \, \Vert g \Vert_M}{N}
\sum_{j=0}^{N-2} | f(j+1) - f(j)|\;.
\end{equation*}
Applying this identity to $\ell = \epsilon N$, $f=\lambda_r
\psi_{t-r}$ and $g(j) = \mf h_j$, by Lemma \ref{s06}, we may replace
in the quadratic variation of $M^N_s(t,G)$ the term $\lambda_r(j)^2
\psi_{t-r}(j)^2$ by an average of these quantities over a macroscopic
interval of length $\epsilon N$, paying the price of an error which
vanishes in $L^1(\bb P_{\mu_N})$, as $N\uparrow\infty$ and then
$\epsilon \downarrow 0$. A summation by parts yields that
\begin{equation*}
\<M^N(t,G)\>_s \;=\; \int_{0}^{s} \frac{E^2}{\gamma^2N} 
\sum_{j = \epsilon N}^{(1-\epsilon) N}\,
\lambda_r(j)^2 \, \psi_{t-r}(j)^2 \, V_{j, \epsilon N} (\eta_r) \,
dr \;+\; O(\epsilon) \;,
\end{equation*}
where $V_{j, \epsilon N} (\eta) = (2 \epsilon N +1)^{-1} \sum_{|k|\le
  \epsilon N} \mf h_{j+k}(\eta)$. By Lemma \ref{replacement} below, we
may replace $V_{j, \epsilon N} (\eta_r)$ by $2\rho_r(j/N)
[1-\rho_r(j/N)] = 2 \sigma(\rho_r(j/N))$ with an error of the same
type.

Up to this point we proved that
\begin{equation*}
\<M^N(t,G)\>_s \;=\; 2 \int_{0}^{s} \frac{E^2}{\gamma^2N} 
\sum_{j = \epsilon N}^{(1-\epsilon) N}\,
\lambda_r(j)^2 \, \psi_{t-r}(j)^2 \, \sigma(\rho_r(j/N)) \,
dr \;+\; O(\epsilon) \;+\; R_{N,\epsilon}\; ,
\end{equation*}
where $R_{N,\epsilon}$ is an error which vanishes in $L^1(\bb
P_{\mu_N})$, as $N\uparrow\infty$ and then $\epsilon \downarrow
0$. Note that the first term on the right hand side is
deterministic. 

By Lemma \ref{s07}, $\lambda_s$ converges to $K_s$, and, by
\eqref{f2.20}, $\psi_s$ converges to $F_s = \partial_x
G_{t-s}/K_{t-s}$ uniformly in time and space. Since $K_r^2 F_{t-r}^2 =
(\partial_x G_{r})^2$ and since $\gamma$ converges to $E$, the lemma
is proved.
\end{proof}

\smallskip

\begin{lemma}
\label{s2.3}
For each function $G$ in $C^{2}_0([0,1])$ and $t>0$, there exists a
finite constant $C_0$, depending only on $G$ and $t$, such that for
all $N\ge 1$,
\begin{equation*}
\bb E_{\mu_N} \big[ \<M^N(t,G)\>_t^2 \big] \;\le \; C_0\;,
\quad
\bb E_{\mu_N} \big[ \sup_{0\le s\le t} M^N_s(t,G)^4
\big] \;\le \; C_0\;.
\end{equation*}
\end{lemma}

\begin{proof}
  We first estimate the quadratic variation $\<M^N(t,G)\>_s $, given
  by \eqref{f2.9}. By \eqref{f2.19}, the solution $\psi_s$ of equation
  \eqref{f2.10} is uniformly bounded.  As the cylinder function $\mf
  h_j$ is also bounded, $\<M^N(t,G)\>_s $ is less than or equal to
\begin{equation*}
C_0 \int_{0}^{s} \frac{1}{N} \sum_{j\in\Lambda_{N}}\, \xi_r(j)^2 \, dr \;.
\end{equation*}
The first assertion of the lemma follows therefore from Lemma
\ref{s09} with $n=2$.

We turn to the second assertion of the lemma. By the
Burkholder-Davis-Gundy inequality and by \cite[Lemma 3]{dg}, the second
expectation appearing in the statement of the lemma is bounded above
by 
\begin{equation*}
C_0 \Big\{ \bb E_{\mu_N} \big[ \<M^N(t,G)\>_t^2 \big] \;+\;
\bb E_{\mu_N} \big[ \sup_{0\le s\le t} |M^N_s(t,G)-M^N_{s-}(t,G)|^4
\big]\Big\} 
\end{equation*}
for some finite constant $C_0$.  In view of the first part of the
proof, it remains to estimate the fourth moment of the jumps.
Clearly, $|M^N_s(t,G)-M^N_{s-}(t,G)| = |{J}_{s}^{N} (g_s) -
{J}_{s_{-}}^{N} (g_s)|$. By the definition of $J^N_s$ and of $\psi_s$,
since $|\xi_{s-}(j)/\xi_{s}(j)\big| \le e^{-\gamma/N}$, and since
$\psi_s$ is uniformly bounded, this latter quantity is less than or
equal to
\begin{equation*}
\frac{1}{\gamma \sqrt N} \sum_{j=0}^{N-1} |(\psi_{t-s})(j)|
\, \big| \xi_s(j) -  \xi_{s_{-}}(j) \big| \;\leq\;
\frac{C_0 }{N^{3/2}} \sum_{j=0}^{N-1} \xi_s(j) \;.
\end{equation*}
The second assertion of the lemma follows from Schwarz inequality and
from Lemma \ref{s09}.
\end{proof}

\begin{lemma}
\label{s2.6}
Fix $G$ in $C^5_0([0,1])$ and $t>0$.  The sequence of martingales
$M^N_s(t,G)$ introduced in \eqref{f2.5} converges in $D([0,t], \bb R)$
to a mean-zero, continuous martingale, denoted by $M_s(t,G)$. For
$G_1$, $G_2$ in $C^5_0([0,1])$, $t_1, t_2>0$, and $0\le s_j\le t_j$,
the covariations of $M_{s_1}(t_1,G_1)$ and $M_{s_2}(t_2,G_2)$ are
given by
\begin{equation*}
\bb E [M_{s_1}(t_1,G_1) \, M_{s_2}(t_2,G_2)] \;=\; 
2 \int_{0}^{s_1\wedge s_2} \int_0^1 \sigma(\rho(r,x))\, (\partial_ x
T_{t_1,r} G_1)(x)\, (\partial_ x T_{t_2,r} G_2)(x) \, dx \, dr\;.
\end{equation*}
\end{lemma}

\begin{proof}
The proof of the convergence in $D([0,t], \bb R)$ of the martingales
$M^N_s(t,G)$ to a mean-zero, continuous martingale, whose quadratic
variation is given by the right hand side of the displayed equation
appearing in the statement of the lemma with $G_j=G$ and $t_j=t$,
relies on \cite[Theorem VIII.3.12]{js}. We claim
that conditions (3.14) and b-(iv) are fulfilled. 
Condition [$\gamma_5$-D] (defined in 3.3 page 470 of \cite{js})
follows from Lemma \ref{s2.5}. By Assertion VIII.3.5 in \cite{js},
condition [$\hat\delta_5$-D] and condition (3.14) are a consequence of 
\begin{equation*}
\lim_{N\to\infty} \bb E_{\mu_N} \Big[ \sup_{s\le t} \big| M^N_s(t,G) -
M^N_{s-}(t,G)\big| \Big]\;=\; 0\;,
\end{equation*}
an assertion which has been proved in the previous lemma.

It remains to prove the formula for the covariances. Fix $G_1$, $G_2$
in $C^5_0([0,1])$, $t_1, t_2>0$, $0\le s_j\le t_j$, and let $s=s_1
\wedge s_2$. Since $M^N_s(t_j,G_j)$, $0\le s\le t_j$, are martingales
in $L^2(\bb P_{\mu_N})$, $\bb E_{\mu_N} [ M^N_{s_1}(t_1,G_1)\,
M^N_{s_2}(t_2,G_2) ] = \bb E_{\mu_N} [ M^N_{s}(t_1,G_1)\,
M^N_{s}(t_2,G_2)]$. By the polarization identity, the computation of
the covariance is reduced to the computation of the variance of the
martingales $M^N_{s}(t_1,G_1)\pm M^N_{s}(t_2,G_2)$. In view of
\eqref{f2.5}, the martingale $M^N_{s}(t_1,G_1)\pm M^N_{s}(t_2,G_2)$
can be represented as a martingale $M^N_{s}(t_1, t_2,G_1,G_2)$. The
proof of Lemma \ref{s2.5} shows that the quadratic variation of this
martingale converges in $L^1(\bb P_{\mu_N})$ to
\begin{equation}
\label{f2.13}
2 \int_{0}^{s} \int_0^1 \sigma(\rho(r,x))\, [(\partial_ x
T_{t_1,r} G_1 \pm T_{t_2,r} G_2)(x)]^2 \, dx \, dr\;.
\end{equation}
By the first part of the proof, the martingale $M^N_{s}(t_1,G_1)\pm
M^N_{s}(t_2,G_2)$ converges in distribution to the martingale
$M_{s}(t_1,G_1)\pm M_{s}(t_2,G_2)$.  As the limit is continuous,
the convergence in the Skorohod topology entails convergence in
distribution at fixed times.  Since, by Lemma \ref{s2.3},
$M^N_{s}(t_1,G_1)\pm M^N_{s}(t_2,G_2)$ is bounded in $L^4(\bb
P_{\mu_N})$,
\begin{equation*}
\bb E\Big[ \big\{ M_{s}(t_1,G_1)\pm M_{s}(t_2,G_2)\big\}^2 \Big]
\;=\; \lim_{N\to\infty} 
\bb E_{\mu_N}\Big[ \big\{ M^N_{s}(t_1,G_1)\pm M^N_{s}(t_2,G_2)\big\}^2 \Big]
\end{equation*}
which completes the proof of the lemma since the right hand side
converges to \eqref{f2.13}.
\end{proof}

We conclude this section stating a result which permits to replace
cylinder functions by functions of the empirical measure. Denote by
$\nu_{\rho}$, $0\le \rho\le 1$, the Bernoulli product measure on
$\{0,1\}^{\bb Z}$ with density $\rho$. For a function $\mf h: \{0,1\}^{\bb
  Z} \rightarrow{\mathbb{R}}$ which depends only on a finite number of
sites, let $\widehat{\mf h}(\rho)=E_{\nu_{\rho}}[\mf h(\eta)]$. Denote by $\tau_j
\eta$, $j\in \bb Z$, $\eta\in \{0,1\}^{\bb Z}$, the configuration
$\eta$ translated by $j$: $(\tau_j \eta)(k) = \eta(j+k)$, $k\in \bb
Z$.  For a cylinder function $\mf h$, whose support is represented by
$\Lambda \subset \bb Z$, and for a configuration $\eta\in\Sigma_N$ the
meaning of $\mf h(\tau_j\eta)$ is clear provided $j+\Lambda \subset \{1,
\dots, N\}$.

\begin{lemma}
\label{replacement}
Let $\{\mu_N : N\ge 1\}$ be a sequence of probability measures in
$\Sigma_N$.  For every continuous function
$G:\mathbb{R}_+\times[0,1]\rightarrow{\bb{R}}$ and every cylinder
function $\mf h$,
\begin{equation*}
\limsup_{N\rightarrow{+\infty}}\mathbb{E}_{\mu_N}\Big[\int_{0}^t\Big|
\frac{1}{N}\sum_{j}
G(s,j/N) \, \mf h(\tau_j\eta_s) \;-\; \int_{0}^{1} G(s,x) 
\, \widehat{\mf h} (\rho(s,x))\, dx\Big| \, ds\Big]\;=\; 0\;,
\end{equation*}
where $\rho(s,x)$ is the solution of the hydrodynamic equation
\eqref{vBe} and where the sum over $j$ is carried over all $j$'s for
which the support of $\mf h$ is contained in $\Sigma_N-j$.
\end{lemma}

The proof of this result is similar to the one presented in \cite{kl},
given the estimate presented in \cite[Lemma 3.1]{blm1}.

\section{Tightness of the density field}
\label{sec2}

We prove in this section that the sequence $\{{Y}_t^N : N\ge 1\}$ is
tight in $D(\bb R_+, \mc H_{-k})$ for $k>7/2$. Recall from Section
\ref{neq} the definition of the eigenfunctions $\{e_n : n\geq 1\}$
and of the eigenvalues $\{\lambda_n : n\geq 1\}$ of the operator
$-\Delta$ defined on $C^2_0([0,1])$. Denote by $\Vert
\,\cdot\,\Vert_{-k}$ the norm of $\mc H_{-k}$, defined as
\begin{equation*}
\| f \|^2_{-k} \;=\; \sum_{n\geq 1} \lambda_n^{-2k} \, \<f,e_n\>^2\;. 
\end{equation*}

By Propositions \ref{s13}, \ref{s10}, and by \eqref{r1}, to prove that
the sequence $\{{Y}_t^N : N\ge 1\}$ is tight it is enough to show that
the sequence $\{J_t^N : N\ge 1\}$ is tight: We claim that for every
$k>7/2$, $T>0$, $\epsilon>0$,
\begin{equation*}
\lim_{A\rightarrow{\infty}}\limsup_{N\rightarrow{\infty}}
\mathbb{P}_{\mu_N}\Big[\sup_{0\leq{t}\leq{T}}\|{{J}}_{t}^{N}\|_{-k}>A\Big]=0
\;, 
\quad \lim_{\delta\rightarrow{0}}\limsup_{N\rightarrow{\infty}}
\mathbb{P}_{\mu_N}\Big[\omega_{\delta}({{J}}_{t}^{N})\geq{\epsilon}\Big]=0\;, 
\end{equation*}
where, for $\delta>0$,
\begin{equation*}
\omega_{\delta}({J}_{t}^{N}) \;=\; \sup_{\substack{|s-t|<\delta\\0\leq{s,t}\leq{T}}}
\|{{J}}_{t}^{N}-{{J}}_{s}^{N}\|_{-k}\;. 
\end{equation*}
The first condition in the penultimate displayed equation is a
consequence of part (a) of Corollary \ref{s2.11}. The second condition
follows from part (b) of that corollary and from Lemma \ref{s2.12}. 

\begin{lemma} 
\label{th:lemmatight}
There exists a finite constant $C_0$, such that for every $n\geq{1}$,
\begin{equation*}
\mathbb{E}_{\mu_N}\Big[\sup_{0\leq{t}\leq{T}}\Big\<{J}_{t}^{N},
\frac{1}{\gamma\lambda_t}\nabla_N^+ e_n\Big\>^{2}\Big]\;\leq\;
C_0\, n^6\; \cdot
\end{equation*}
\end{lemma}

\begin{proof}
By \eqref{f2.7} and \eqref{f2.14}, 
\begin{equation*}
{J}_{t}^{N}(e_n) \;=\; {J}_{0}^{N}(e_n) \;+\;
\int_{0}^{t}{J}_{s}^{N}(\ms A_{s} e_n) \, ds \;+\; 
\mathcal{M}^N_{t}(e_n)\;, 
\end{equation*}
where $\mathcal{M}^N_{t}(e_n) $ is the martingale appearing on the right
hand side of \eqref{f2.14} with $g=e_n$. We estimate separately each
term of the previous expression. By Schwarz inequality,
\begin{equation*}
\mathbb{E}_{\mu_N}\big[{J}_{0}^{N}( e_n)^2\big] \;\le\;
\frac{1}{\gamma^2} \sum_{j=0}^{N-1}
\frac{(\nabla_N^+ e_n)(j/N)^2}{\lambda_0(j)^2} \,
\mathbb{E}_{\mu_N} \big[\{\xi_0(j)-\lambda_0(j)\}^2 \big]\;.
\end{equation*}
By assumption \eqref{08}, the expectation is bounded by $C_0/N$.
Hence, since $\lambda_0$ is bounded below by a strictly positive
constant, the previous sum is less than or equal to $C_0 n^2$.

We turn to the time integral term in the decomposition of
$J_{t}^{N}(e_n)$. By Schwarz inequality, and by the definition of
$J^N_t$,
\begin{equation*}
\begin{split}
\mathbb{E}_{\mu_N}\Big[\sup_{0\leq{t}\leq{T}} & \Big(\int_{0}^t{J}_{s}^{N}(\ms
A_s^Ne_n)\,ds\Big)^2\Big] \\
 &\le\; T \int_0^T  \frac{1}{\gamma^2 N}\sum_{j=0}^{N-1}
\frac{[\nabla_N^+ (\ms A_s e_n)](j)^2}{\lambda_s(j)^2} \,
\varphi_s(j,j) \, ds \\
& +\; T \int_0^T \frac{1}{\gamma^2  N}\sum_{j\not = k} 
\frac{[\nabla_N^+ (\ms A_s e_n)](j) \, [\nabla_N^+ (\ms A_s
  e_n)](k)} {\lambda_s(j)\lambda_s(k)} \, \varphi_s(j,k)\, ds\;,
\end{split}
\end{equation*}
where $\varphi_s(j,j) = \mathbb{E}_{\mu_N}
[\{\xi_s(j)-\lambda_s(j)\}^2 ]$, $\varphi_s(j,k) = \mathbb{E}_{\mu_N}
[\{\xi_s(j)-\lambda_s(j)\}\, \{\xi_s(k)-\lambda_s(k)\} ]$. Recall that
$\lambda_s(j)$ is bounded below by a strictly positive constant. By
Lemma \ref{sd4}, $\sup_{0\le s\le T} \max_{j,k} |\varphi_s(j,k)| \le
C_0/N$. On the other hand, in view of Lemma \ref{s2.10}, by a Taylor
expansion and since $(\ms A_s e_n)(0) = (\ms A_s e_n)(N)=0$,
\begin{equation*}
\begin{split}
& \sup_{0\le s\le T} \max_{1\le j\le N-2} \big|\, [\nabla_N^+ (\ms A_s
e_n)](j) \, \big| \;\le\; C_0 n^3\;, \\
& \quad \sup_{0\le s\le T} \, \max_{k=0, N-1} \big|\, [\nabla_N^+ (\ms A_s
e_n)](k) \, \big| \;\le\; C_0 n^2 N \;.
\end{split}  
\end{equation*}
It follows from these bounds that the penultimate displayed equation
is bounded by $C_0 n^6$.

It remains to examine the martingale term in the decomposition of
$J_{t}^{N}(e_n)$. By the definition \eqref{f2.14} of the martingale
$\mathcal{M}^N_{t}(e_n)$, by Doob's inequality, and by \eqref{VQuad},
\begin{equation*}
\mathbb{E}_{\mu_N}\Big[\sup_{0\leq{t}\leq{T}}
\mathcal{M}^N_{t}(e_n)^{2}\Big] 
\; \le\; \mathbb{E}_{\mu_N} \Big[ 
\int_0^T  \frac{4E^2}{\gamma^2 N}\sum_{j=0}^{N-1}
\frac{(\nabla_N^+ e_n)(j)^2}{\lambda_s(j)^2} \, 
\xi_s(j)^2 \, \mf h_j(\eta_s) \, ds \Big]\;.
\end{equation*}
Since the cylinder functions $\mf h_j$ are bounded and since, by Lemma
\ref{s08}, $\lambda_s$ is uniformly bounded below, by Lemma \ref{s09}
this expression is less than or equal to $C_0 n^2$.  This completes
the proof of the lemma.
\end{proof}

\begin{corollary}
\label{s2.11}
For each $k>7/2$
\begin{equation*}
\begin{split}
&(a)\quad \limsup_{N\rightarrow{+\infty}}
\mathbb{E}_{\mu_N}\Big[\sup_{0\leq{t}\leq{T}}\|{{J}}^N_{t}\|_{-k}^{2}\Big]
<{\infty}\\
&(b)\quad
\lim_{m\rightarrow{+\infty}}\limsup_{N\rightarrow{+\infty}}
\mathbb{E}_{\mu_N}\Big[\sup_{0\leq{t}\leq{T}}\sum_{n\geq{m}}
\<{{J}}^N_{t},e_n\>^{2}\lambda_{n}^{-2k}\Big]=0.
\end{split}
\end{equation*}
\end{corollary}

\begin{proof}
This result is a consequence of the previous lemma and of the observation
that 
\begin{equation*}
\sup_{0\leq{t}\leq{T}}\|{{J}}^N_{t}\|_{-k}^{2} \;\le\;
\sum_{n\geq 1}\lambda_{n}^{-2k} \, \sup_{0\leq{t}\leq{T}} 
\big| {J}^N_{t}(e_n) \big|^{2}\;.
\end{equation*}
\end{proof}

\begin{lemma}
\label{s2.12}
For every $n\ge 1$ and every $\epsilon>{0}$,
\begin{equation*}
\lim_{\delta\rightarrow{0}}\limsup_{N\rightarrow{+\infty}}\mathbb{P}_{\mu_N}
\Big[\sup_{\substack{|s-t|<\delta\\0\leq{s,t}\leq{T}}}\quad
[{J}^N_{t}(e_{n}) - {{J}}^N_{s}(e_{n})]^{2} > 
\epsilon \, \Big]\;=\; 0\;.
\end{equation*}
\end{lemma}

\begin{proof}
Recall the decomposition of $J^N_{t}(e_{n})$ presented at the
beginning of the proof of Lemma \ref{th:lemmatight}. We first claim
that for every $\epsilon>0$,
\begin{equation}
\label{f2.15}
\lim_{\delta\rightarrow{0}} \limsup_{N\rightarrow{+\infty}}
\mathbb{P}_{\mu_N}\Big[\sup_{\substack{|s-t|<\delta\\0\leq{s,t}\leq{T}}}\quad
\big| \, \mathcal{M}_{t}^{N}(e_n) - \mathcal{M}_{s}^{N}(e_n)\, \big| > 
\epsilon\, \Big]\;=\; 0\;. 
\end{equation}

Denote by $\omega'_{\delta}(\bs x)$ the modified modulus of continuity
of a path $\bs x$ in $D([0,T], \bb R)$. Since $\omega_{\delta}(\bs x)
\leq 2\omega'_{\delta}(\bs x) + \sup_{t\le T}|\bs x_t - \bs x_{t-}|$,
to prove \eqref{f2.15} it is enough to show that for every $\epsilon>0$
\begin{equation}
\label{f2.16}
\begin{split}
& \lim_{\delta\rightarrow{0}} \limsup_{N\rightarrow{+\infty}}
\mathbb{P}_{\mu_N}
\big[\omega'_{\delta}(\mathcal{M}_{t}^{N}(e_n)) > \epsilon \,\big]\;=\; 0\;,
\\
&\quad 
\limsup_{N\rightarrow{+\infty}} \mathbb{P}_{\mu_N} \big[\sup_{t\le T}
|\mathcal{M}_{t}^{N}(e_n)-\mathcal{M}_{t_{-}}^{N}(e_n)|>\epsilon\,
\big]\;=\; 0\;.
\end{split}
\end{equation}

Clearly, $|\mathcal{M}_{t}^{N}(e_n) - \mathcal{M}_{t_{-}}^{N} (e_n)| =
|{J}_{t}^{N} (e_n) - {J}_{t_{-}}^{N} (e_n)|$. By definition of $J^N_t$
and since $|\xi_{t-}(j)/\xi_{t}(j)\big| \le e^{-\gamma/N}$ this
latter quantity is less than or equal to
\begin{equation*}
\frac{1}{\sqrt N} \sum_{j=0}^{N-1} \frac{ |(\nabla_N^+e_n)(j)|}
{\lambda_t(j)} \, \big| \xi_t(j) -  \xi_{t_{-}}(j) \big| \;\leq\;
\frac{C_0 n}{N^{3/2}} \sum_{j=0}^{N-1} \xi_t(j) \;.
\end{equation*}
The second condition of \eqref{f2.16} follows from the previous
estimate, from Markov inequality and from the fact that the
expectation of $\xi_t(j)$ (which is equal to $\lambda_t(j)$) is
uniformly bounded.

We turn to the first condition of \eqref{f2.16}.  By Aldous criterium,
it is enough to show that for every $\epsilon>0$
\begin{equation*}
\lim_{\delta\rightarrow{0}} \limsup_{N\rightarrow{+\infty}}
\sup_{\substack{\tau\in{\mathfrak {T}_{\tau}}\\
0\leq{\theta}\leq{\delta}}}\mathbb{P}_{\mu_N}
\Big[|\mathcal{M}_{\tau+\theta}^{N}(e_n)-
\mathcal{M}_{\tau}^{N}(e_n)| > \epsilon\Big]\;=\;0\;,
\end{equation*}
where $\mathfrak {T}_{\tau}$ represents the set of stopping times
bounded by $T$. By Tchebychev inequality and by the explicit
expression for the quadratic variation of $\mathcal{M}^N_t(e_n)$, the
previous probability is bounded by
\begin{equation*}
\mathbb{E}_{\mu_N} \Big[ \int_{\tau}^{\tau+\theta}
\frac{E^2}{\gamma^2 \epsilon^2 N}\sum_{j=0}^{N-1} \xi_s(j)^2 \mf h_j(\eta_s) 
\, \frac{(\nabla_N^+ e_n)(j/N)^2}{\lambda_s(j)^2} \,ds\Big]\;.
\end{equation*}
By Lemma \ref{s09} and Lemma \ref{s08}, the previous expectation is
bounded above by $C_0 n^2 \delta/\epsilon^2$, proving the first
assertion of \eqref{f2.16}. This proves \eqref{f2.15}.

We claim that for every $\epsilon>0$
\begin{equation}
\label{f2.17}
\lim_{\delta\rightarrow{0}}\limsup_{N\rightarrow{+\infty}}
\mathbb{P}_{\mu_N}\Big[\sup_{\substack{|s-t|<\delta\\0\leq{s,t}\leq{T}}}\Big|
\int_{s}^{t}J_r^N(\ms A_r e_n)\,dr\Big|>\epsilon\Big]=0
\end{equation}
By Tchebychev inequality, the previous probability is bounded by
\begin{equation*}
\frac{\delta}{\epsilon^2} \, \mathbb{E}_{\mu_N}\Big[
\int_{0}^{T}\Big(\frac{1}{\gamma \sqrt N}\sum_{j=0}^{N-1}
\frac{\nabla_N^+ (\ms A_r e_n)(j/N)}{\lambda_r(j)} 
\, [\xi_r(j)-\lambda_r(j)] \Big)^2\,dr\Big] \; .
\end{equation*}
The computations performed in the proof of Lemma \ref{th:lemmatight}
yield that the previous expression is bounded by $C_0 n^6
\delta/\epsilon^2$. This proves \eqref{f2.17}.

The assertion of the lemma is a consequence of \eqref{f2.15},
\eqref{f2.17}. 
\end{proof}

\section{Exponential estimates}
\label{sec4}

We present in this section some bounds on the process $\xi_t$. By
\eqref{fh1} and by the definition of the variables $\xi_t(j)$, for $0\leq
j\leq N-2$,
\begin{equation} 
\label{02}
\xi_{t}(j) \;\le\; \xi_{t}(j+1) \;\le\;e^{- \gamma/N}\xi_{t}(j)\;.
\end{equation}

\begin{lemma}
\label{s09}
Fix $n\ge 1$, $T>0$ and a sequence of probability measures $\{\mu_N :
N\geq 1\}$ on $\Sigma_N$. There exists a finite constant $C_1$ and
$N_0\ge 1$, depending only on $n$, $\beta$, $E$ and $T$, such that for
all $0\le j\le N-1$ and all $N\ge N_0$,
\begin{equation*}
\bb E_{\mu_N} \Big[
\sup_{0\le t\le T} \xi_t(j)^n \Big] \;\le\; C_1\;.
\end{equation*}
\end{lemma}

\begin{proof}
Fix $n\ge 1$ and $T>0$. In the proof $C_1$ represents a finite
constant which depends only on $n$, $\beta$, $T$ and $E$ and which may
change from line to line. We first claim that
\begin{equation}
\label{fg5}
\sup_{0\le t\le T}  \; \max_{0\le j\le N-1}
\bb E_{\mu_N} [\xi_t(j)^n ] \;\le\; C_1\;.
\end{equation}

A similar computation to the one performed just after \eqref{fg1}
shows that  for each $0\le j\le N-1$
\begin{equation}
\label{fg4}
\xi_t(j)^n =\xi_0(j)^n \;+\; \int_0^t 
\Big\{ [\Omega_n \, \xi_s^n]\, (j) + A_n (s,j) \Big\}
\, ds \; +\; \mathcal{M}^N_n(t,j) \;.
\end{equation}
In this formula, $\mathcal{M}^N_n(\cdot,j)$ is a zero-mean martingale;
$\Omega_n$ is the linear operator equal to $\Omega$ in the interior of
$\Lambda_N$ and given at the boundary by
\begin{equation}
\label{fg3}
\left\{
\begin{split}
& (\Omega_n f)(0) \;=\; - \alpha \, N \, R_n\,
f(0) \;+\; N (\nabla_N^+ f) (0)\;, \\
& \vphantom{\Big\{}
(\Omega_n f)(N-1) \;=\; \beta N S_n
f(N-1) \;-\; N \Big(1+ \frac{E}{N}\Big) (\nabla_N^- f) (N-1)\;,
\end{split}
\right.
\end{equation}
where
\begin{equation*}
R_n \;=\; N \Big(1+ \frac{E}{N}\Big)
\Big(1-e^{n\gamma/N} \Big) \;,\quad
S_n \;=\; N  \Big(e^{-n\gamma/N} -1\Big)\; ;
\end{equation*}
and 
\begin{equation*}
\label{fg4.1}
A_n(t,j) \;=\; - N^2 \Big\{\Big(1+\frac{E}{N}\Big)(e^{\gamma n/N}-1)
+ (e^{-\gamma n/N}-1) \Big\}\, \xi_t(j)^n \, \eta_t(j)\, \eta_t(j+1) \;.
\end{equation*}
Notice that $A_1(t,j)=0$ and that $R_1=S_1=E$ so that
$\Omega_1=\Omega$.

It follows from the previous computations that $f_n(t,j) =\bb
E_{\mu_N} [\xi_t(j)^n ]$ satisfies the differential inequality
\begin{equation*}
\partial_t f(t,j) \;\le\; (\Omega_n f) (t,j) \;.
\end{equation*}

Let $F_n(t,\cdot)$ be the solution of equation \eqref{fg9}, with
$\Omega_n$ instead of $\Omega$ and initial condition $F_n(0,j) =
f_n(0,j)$. By the maximum principle, $f_n(t, \cdot) \le
F_n(t,\cdot)$ for all $t\ge 0$. Claim \eqref{fg5} follows from Lemma
\ref{s05} and the bound $F_n(0,j)\leq{\exp\{-\gamma n\}}$.

It remains to bring the supremum inside the expectation. Since, by
\eqref{02}, $\xi_t(j)$ is increasing in $j$, it is enough to prove the
lemma for $j=N-1$. However, by \eqref{02}, $\xi_t(N-1) \le e^{-
  \gamma} \xi_t(j)$ so that
\begin{equation*}
\bb E_{\mu_N} \Big[
\sup_{0\le t\le T} \xi_t(N-1)^n \Big] \;\le\;
e^{-\gamma n} \, \bb E_{\mu_N} \Big[
\sup_{0\le t\le T} \frac 1N \sum_{j=0}^{N-1} \xi_t(j)^n \Big]\; .
\end{equation*}
By \eqref{fg4},
\begin{equation*}
\xi_t(j)^n \; \leq\;  \xi_0(j)^n \;+\; \int_0^t [\Omega_n \xi_s^n](j)
\,ds  \; +\; \mathcal{M}^N_n(t,j)\;.
\end{equation*}
We need therefore to estimate three terms. The first one is given by
\begin{equation*}
\frac 1N \sum_{j=0}^{N-1} \xi_0(j)^n \;\le\; e^{-\gamma n}\; .
\end{equation*}
The second one is also simple to handle. Since
\begin{equation*}
\frac 1N \sum_{j=0}^{N-1} [\Omega_n \xi^n](j) \;\le\; E \, \xi(0)^n
\;+\; \beta\, N\, \big(e^{-\gamma n/N} - 1 \big) \, \xi(N-1)^n\;,
\end{equation*}
we have that
\begin{equation*}
 \bb E_{\mu_N} \Big[ \sup_{0\le t\le T}
\int_0^t  \frac 1N \sum_{j=0}^{N-1} [\Omega_n \xi_s^n](j) \, ds \Big] \;\le\; C_1 \bb E_{\mu_N}
\Big[ \int_0^T \big\{ \xi_s(0)^n + \xi_s(N-1)^n\}\, ds  \Big]\;.
\end{equation*}
By \eqref{fg5}, this expression is bounded by a constant independent
of $N$. To estimate the martingale term, apply Doob's inequality and
use the fact that the martingales $\mathcal{M}^N_n(t,\cdot)$ are
orthogonal to get that
\begin{equation*}
\bb E_{\mu_N} \Big[ \Big(\sup_{0\le t\le T}
\frac 1N \sum_{j=0}^{N-1} \mathcal{M}^N_n(t,j) \Big)^2 \Big] \;\le\;
\bb E_{\mu_N} \Big[ \int_0^T \frac {C_1 }{N^2}
\sum_{j=0}^{N-1} \xi_t(j)^{2n} \, dt \Big]\; .
\end{equation*}
By \eqref{fg5}, this expression is bounded by $C_1N^{-1}$, which
concludes the proof of the lemma.
\end{proof}

\begin{lemma}
\label{sd4}
Let $\{\mu_N: N\geq 1\}$ be a sequence of measures on $\Sigma_N$
satisfying \eqref{08}.  Then, for each fixed $T>0$, there exist finite
constants $C_1$ and $N_0\ge 1$, depending only on $E$, $\beta$, $T$
and $A_2$ such that
\begin{equation*}
\sup_{0\le t\le T} \max_{j\in\Lambda_N} \bb E_{\mu_N} \Big[
\big(\xi_t(j) - \lambda_t(j)\big)^4 \Big] \;\le\;
\frac{C_1}{N^2}\;\cdot
\end{equation*}
\end{lemma}

\begin{proof}
For  $0\le k\le N-1$ and $t\ge 0$, let $q_t(k,\cdot)$ be the
solution of equation \eqref{fg9} with initial condition $q_0(k,j) =
\delta_{k,j}$. By \eqref{fg2},
\begin{equation*}
\xi_t(j) \;=\; \sum_{k=0}^{N-1} \xi_0(k) q_t(k,j) \;+\;
\sum_{k=0}^{N-1} \int_0^t  q_{t-s}(k,j) \, d\mathcal{M}^N_s(k)\;, 
\end{equation*}
so that
\begin{equation}
\label{fg10}
\xi_t(j) - \lambda_t(j) \;=\; \sum_{k=0}^{N-1}
\Big(\xi_0(k) - \lambda_0(k)\Big) q_t(k,j) \;+\;
\sum_{k=0}^{N-1} \int_0^t  q_{t-s}(k,j) \, d\mathcal{M}^N_s(k)\;.
\end{equation}
To prove the lemma we need to estimate the fourth moments of
the terms on the right hand side of \eqref{fg10}.

By H\"older's inequality,
\begin{equation*}
\begin{split}
& E_{\mu_N} \Big[ \Big(
\sum_{k=0}^{N-1} \Big(\xi_0(k) - \lambda_0(k)\Big)
\, q_t(k,j) \Big)^4 \Big] \\
&\qquad \le\;
E_{\mu_N} \Big[
\sum_{k=0}^{N-1} \Big(\xi_0(k) - \lambda_0(k)\Big)^4
q_t(k,j)  \Big] \Big( \sum_{k=0}^{N-1} q_t(k,j) \Big)^3.
\end{split}
\end{equation*}
Notice that
\begin{equation*}
| \xi_0(k) - \lambda_0(k) | \;\le\;
\frac {C_1 }{N} \Big| \sum_{j=1}^k \Big\{ \eta_0(j)
- \rho_0\Big(\frac{j}{N}\Big) \Big\} \Big|
\end{equation*}
for some finite constant $C_1$ which depends only on $E$, $\beta$,
$T$, $A_2$, and whose value may change from line to line. Therefore,
by assumption \eqref{08} and since, by \eqref{f2.3},
$\sum_{k=0}^{N-1}q_s(k, j)$ is uniformly bounded in $j$ and $0\le s\le
T$, the fourth moment of the first term on the right hand side of
\eqref{fg10} is bounded by $C_1/N^2$.

We turn to the martingale term in \eqref{fg10}.  For $0\le r\le t$,
let $\mathcal{M}^N_{j,t} (r)$ be the martingale defined by
\begin{equation*}
\mathcal{M}^N_{j,t}(r) \;=\; \sum_{k=0}^{N-1} \int_0^r
q_{t-s}(k,j) \, d\mathcal{M}^N_s(k)\;.
\end{equation*}
By the Burkholder-Davis-Gundy inequality and \cite[Lemma 3]{dg}, there
exists a finite constant $C_0$ such that
\begin{equation*}
\bb E_{\mu_N} \big[ \mathcal{M}^N_{j,t}(t)^4  \big]
\;\leq\; C_0 \Big\{ \bb E_{\mu_N} \big[ \<\mathcal{M}^N_{j,t}\>_t^2
\big] \;+\; \bb E_{\mu_N} \big[ \sup_{0\le s\le t}
|\mathcal{M}^N_{j,t}(s) - \mathcal{M}^N_{j,t}(s-) \big|^4\,
\big]  \Big\} \;,
\end{equation*}
where $\<\mathcal{M}^N_{j,t}\>_r$ stands for the quadratic variation
of the martingale $\mathcal{M}^N_{j,t}$.

We first estimate the jump term. By \eqref{fg10} and by definition of $\xi_s$,
$|\mathcal{M}^N_{j,t}(s) - \mathcal{M}^N_{j,t}(s-) \big| = |\xi_s(j) -
\xi_{s-}(j)| \le (C_0/N) \xi_s(j)$. Hence, by Lemma \ref{s09}, the
second expectation on the right hand side of the previous formula is
bounded above by $C_0/N^4$.

It remains to examine the quadratic variation.  By \eqref{VQuad} the
quadratic variation of the martingale $\mathcal{M}^N_{j,t}(r)$ is
bounded above by
\begin{equation*}
\begin{split}
& C_1 \int_0^r \sum_{k=0}^{N-1} q_{t-s}(k, j)^2 \, \xi_s(k)^2 \, ds \\
&\quad \le \; C_1  \int_0^r \max_{0\leq k\leq N-1}q_{t-s}(k,j)\, 
\sum_{k=0}^{N-1} q_{t-s}(k, j) \, \xi_s(k)^2 \, ds\;.
\end{split}
\end{equation*}
By remark \eqref{f2.3}, $\sum_{k=0}^{N-1}q_s(k, j)$ is uniformly
bounded in $j$ and $0\le s\le T$, and by Corollary \ref{s03},
$\max_{0\le k \le N-1} q_{t-s}(k, j)$ is bounded above by $C_1
\{N^2(t-s)\}^{-1/2}$ for all $N$ large enough and all $j$.  Since, by
\eqref{02}, $\xi_s(k) \le \xi_s(N-1)$, $0\leq k \leq{N-1}$, the
previous expression is less than or equal to
\begin{equation*}
C_1  \int_0^r \frac{1}{N \sqrt{t-s}}\, \xi_s(N-1)^2 \, ds \;.
\end{equation*}
Hence, by the Cauchy-Schwarz's inequality,
\begin{equation*}
\bb E_{\mu_N}
\Big[ \<\mathcal{M}^N_{j,t}\>_t^2  \Big] \;\le\;
\frac {C_1}{N^2} \, \bb E_{\mu_N}
\Big[\int_0^t \frac{1}{\sqrt{t-s}} \, \xi_s(N-1)^4 \,ds \Big]\;,
\end{equation*}
which concludes the proof of the lemma in view of Lemma \ref{s09}.
\end{proof}

\section{The operators $\Omega_n$}
\label{sec11}

We prove in this section some properties of the solutions of the
differential equation $\partial_t f_t = \Omega_n f_t$, where
$\Omega_n$ is the linear operator defined by \eqref{fg11} and
\eqref{fg3}. We start with a result on classical solutions of the
viscous Burgers equation \eqref{vBe}.

\begin{lemma} 
\label{s2.13}
Let $\rho_0$ be density profile in $C^4([0,1])$. Then, the solution of
the viscous Burgers equation \eqref{vBe} belongs to
$C^{2,3}([0,\infty)\times [0,1])$ and the solution of the linear
equation \eqref{eq_linearized} belongs to $C^{2,4}([0,\infty)\times
[0,1])$.
\end{lemma}

\begin{proof}
Since $\rho_0$ belongs to $C^4([0,1])$, $K_0$ defined by
\eqref{eq_linearized} belongs to $C^{2m+1}([0,1])$ with
$m=2$. Therefore, the (generalized) Fourier series expansion of the
solution $K$ of \eqref{eq_linearized} with initial condition $K_0$,
provided by the method of separation of variables, yields that $K \in
C^{m,2m}([0,\infty)\times [0,1])$. Moreover, since the semigroup
corresponding to \eqref{eq_linearized} is positivity improving and
since $K_0$ is bounded below by a positive constant, so is
$K_t$. Thus, $\rho(t,x)=\partial_x K/E K$, which solves the viscous
Burgers equation, is well defined and belongs to
$C^{2,3}([0,\infty)\times [0,1])$. Uniquenes of classical solutions of
\eqref{vBe} completes the proof.
\end{proof}

\noindent{\bf Note:} With the same notation as in the previous lemma,
assume that $K_0$ belongs to $C^{2m+2}([0,1])$, $m\ge 0$, so that
$\partial_x K_0 \in C^{2m+1}([0,1])$. Since $\partial_x K$ satisfies
the same equation as $K$, one obtains by the previous argument that
$\partial_x K \in C^{m,2m}([0,\infty)\times[0,1])$, so that $K \in
C^{m,2m+1}([0,\infty)\times[0,1])$. \smallskip

We turn to the operator $\Omega_n$, which should be understood as a
small perturbation of $\Omega_0$, obtained from $\Omega_n$ by setting
$\alpha=\beta=0$, and which represents the generator of a weakly
asymmetric random walk on $\Lambda_N$ with reflection at the boundary.

Let $m_N$ be the measure given by
\begin{equation*}
m_N (k) \;=\; \Big(1+ \frac{E}{N}\Big)^{-k}\;,
\quad\text{$0 \le k\le  N-1$}\;.
\end{equation*}
Denote by $\< \cdot, \cdot\>_{m_N}$ the scalar product in
$L^2(m_N)$. A calculation shows that for each $n\ge 0$, $\Omega_n$ is
self-adjoint in $L^2(m_N)$, that is
\begin{equation*}
\<g, \Omega_n f\>_{m_N} \;=\; \<\Omega_n g, f\>_{m_N}\;, \quad
f, g \in L^2(m_N).
\end{equation*}

For  $p\ge 0$, denote by $\Vert \cdot \Vert_p$, the $L^p$ norm with respect
to $m_N$ and by $D_N$ the Dirichlet form associated to $\Omega_0$ and
$m_N$:
\begin{equation*}
D_N(f) \;=\; \<f, - \Omega_0 f\>_{m_N} \;=\;
N^2 \sum_{k=0}^{N-2} [ f (k+1) - f(k)]^2 \, m_N(k)\;.
\end{equation*}
The logarithmic Sobolev inequality for the weakly asymmetric random
walk on $\Lambda_N$ with reflection at the boundary \cite[Example
3.6]{ds96} states that there exists a finite constant $A_0$, depending
only on $E$, such that
\begin{equation}
\label{fa2}
\sum_{k=0}^{N-1} f (k)^2 \log f(k)^2 m_N(k) \;\le\; A_0 D_N(f)
\end{equation}
for all functions $f$ such that $\Vert f\Vert_2 =1$ and all $N\ge 2$.

Fix $n\ge 1$, an initial condition $f : \Lambda_N \to \bb R$
and denote by $f^{(n)}$ the solution of the linear differential
equation
\begin{equation}
\label{fa1}
\partial_t f_t^{(n)} \;=\; \Omega_n f_t^{(n)}\;, \quad
f_0^{(n)} \;=\; f\; .
\end{equation}
It is not difficult to prove a maximum principle for the solution of
this linear equation,
\begin{equation*}
f_t^{(n)} \;\ge\; 0\;\; \text{for all}\;\;  t\ge 0 \;\;\text{if} \;\;
f\;\ge\; 0 \;,
\end{equation*}
and to deduce the existence of a unique solution.

\begin{lemma}
\label{s06}
Fix $n\ge 1$ and let $f_t = f^{(n)}_t$ be the solution of
\eqref{fa1}. There exists a finite constant $C_0$, depending only on
$E$, $\beta$ and $n$, such that for any $t\geq 0$
\begin{equation*}
\Vert f_t \Vert_2^2 \;+\; \int_0^t D_N (f_s) \, ds \;\le\;
\, e^{C_0 t} \Vert f_0 \Vert_2^2\;.
\end{equation*}
\end{lemma}

\begin{proof}
Fix $n\ge 1$. Differentiating $\Vert f_t \Vert_2^2$ yields 
\begin{equation}
\label{eq2}
\frac 12 \, \frac{d}{ds} \Vert f_t \Vert_2^2 \;=\; - \, \alpha N R_n f_s(0)^2m_N(0)
\;+\; \beta N S_n f_s(N-1)^2 m_N(N-1) \;-\; D_N(f_s) \;.
\end{equation}
For every $1\le m\le N$ and every $s\geq 0$,
\begin{equation}
\label{claim2}
f_s(N-1)^2 \;\le\; 2 e^{-\gamma}\Big( \frac m{N^2} D_N(f_s) \;+\;
\frac 1m \< f_s , f_s \>_{m_N} \Big)\;.
\end{equation}
Indeed, fix $1\le m\le N$. By Young's inequality, 
\begin{equation*}
f_s(N-1)^2  \;\le\; 2\Big(f_s(N-1)-\frac 1m \sum_{k=N-m}^{N-1}f_s(k)\Big)^2
\;+\; 2\Big(\frac 1m \sum_{k=N-m}^{N-1}f_s(k)\Big)^2\;.
\end{equation*}
By Schwarz inequality and since $m_N(k)\geq e^{\gamma}$ for $0\le k
\le N-1$, the second term on the right hand side is less than or equal
to 
\begin{equation*}
\frac 2m \sum_{k=N-m}^{N-1}f_s(k)^2 \;\leq\; 
\frac {2e^{-\gamma}}m \sum_{k=0}^{N-1}f_s(k)^2 \, m_N(k)
\;=\; \frac {2e^{-\gamma}}m \, \< f_s , f_s \>_{m_N}\;.
\end{equation*}
The first term on the right hand side can be rewritten as
\begin{equation*}
2\Big(\frac 1m
\sum_{k=N-m}^{N-1}\sum_{j=k}^{N-2}[f_s(j+1)-f_s(j)]\Big)^2 
\;\le\; 2 \sum_{k=N-m}^{N-1}\sum_{j=k}^{N-2}[f_s(j+1)-f_s(j)]^2 \;.
\end{equation*}
Since $m_N(k)\geq e^{\gamma}$ this sum is bounded above by
\begin{equation*}
2me^{-\gamma}\sum_{j=0}^{N-2}[f_s(j+1)-f_s(j)]^2 \, m_N(j)
\;=\; 2e^{-\gamma} \frac m{N^2} D_N(f_s)\;,
\end{equation*}
which proves \eqref{claim2}.

Set $m = [N e^{\gamma}/4\beta S_n] \wedge N $, where $[a]$ represents
the integer part of $a$. Putting together \eqref{eq2} and
\eqref{claim2} yields
\begin{equation*}
\frac{d}{ds}\< f_s , f_s \>_{m_N} \;\leq\; -D_N(f_s) \;+\; C_0\< f_s ,
f_s \>_{m_N} \;.
\end{equation*}
To conclude the proof it remains to apply Gronwall's inequality.
\end{proof}

Next result shows that the solutions of \eqref{fa1} are monotone.

\begin{lemma}
\label{s02}
Fix $n\ge 1$ and a non-negative initial condition $f_0: \Lambda_N\to
\bb R$ such that $f_0(j) \le f_0(j+1)$, $0\le j <N-1$. Then, the
solution $f_t=f^{(n)}_t$ of \eqref{fa1} conserves the monotonicity:
\begin{equation*}
f_t(j) \;\le\; f_t(j+1)
\end{equation*}
for all $t\ge 0$ and $0\le j <N-1$. Conversely, if the non-negative
initial condition is such that $f_0(j+1) \le e^{- \gamma
n/N}f_0(j)$, $0\le j <N-1$, the same property holds at later times:
\begin{equation*}
f_t(j+1) \;\le\; e^{- \gamma
n/N}\, f_t(j)
\end{equation*}
for all $t\ge 0$ and $0\le j <N-1$.
\end{lemma}

\begin{proof}
For $t>0$, $j\in\{1, \dots, N-1\}$, let $g_t(j)=f_t(j)-f_t(j-1)$.  It
is easy to show that $g_t$ satisfies an equation of the form
\begin{equation}
\label{fa7}
\frac{d}{dt} g_t \;=\; \widetilde{\Omega}_n g_t \;+\; \psi_t\; ,
\end{equation}
where all the entries in $\psi_t$ are null except for the first and the
last which are equal to $\alpha N R_n f_t(0)$ and $\beta N S_n
f_t(N-1)$, respectively.

Moreover, $\widetilde{\Omega}_n$ is a tridiagonal matrix whose
diagonal elements are equal to $-N^2(2+E/N)$, upper off-diagonal elements
equal $N^2$ and lower off-diagonal elements are equal to $N^2(1+E/N)$.

We may now apply the maximum principle to conclude the proof of the
first assertion of the lemma because, as already seen, the solution
$f_t$ is non-negative. Alternatively, we can recall the observation
(see \cite[Exercise 97, pag. 375]{rs}) that for any $t>0$ the
exponential $e^{At}$ of a matrix $A$ has all its entries positive if
and only if all the off-diagonal elements of $A$ are
non-negative. Since that holds for $\widetilde{\Omega}_n$ and
$\Omega_n$, then $g_t$, which can be written as
\begin{equation*}
g_t \;=\; e^{\widetilde{\Omega}_n t} \, g_0
\;+\; \int_0^t e^{\widetilde{\Omega}_n(t-s)} \psi_s\,\ ds\;, 
\end{equation*}
is non-negative.

The same argument applies to the second assertion.  For $t>0$,
$j\in\{1, \dots, N-1\}$, let $g_t(j)= e^{- \gamma n/N}
f_t(j-1)-f_t(j)$.  Then, $g_t$ satisfies the equation \eqref{fa7}
where all the entries in $\psi_t$ are null except for the first and the
last which are equal to $N(N+E)(1-\alpha)(e^{-\gamma n/N}-1) f_t(0)$
and $N^2(1-\beta)(e^{-\gamma n/N}-1) f_t(N-1)$, respectively.
\end{proof}

\begin{lemma}
\label{s05}
Fix $n\ge 1$ and let $f_t = f^{(n)}_t$ be the solution of
\eqref{fa1}. There exists a finite constant $C_0$, depending only on
$E$, $\beta$ and $n$, such that for any $t\geq 0$
\begin{equation*}
\Vert f_t\Vert_M \;\le\; C_0 \, e^{C_0 t}  \Vert f_0\Vert_M \;.
\end{equation*}
for all $t\ge 0$.
\end{lemma}

\begin{proof}
Let $g_0$ be the function which is constant and equal to $\Vert
f_0\Vert_M$ and denote by $g_t$ the solution of \eqref{fa1} with
initial condition $g_0$. By the maximum principle, 
$f_t(j)^2 \le g_t(j)^2$, for all $1\le j\le N$, $t\geq 0$. 

By Lemma \ref{s02}, $e^{n\gamma} g_t(k) \le g_t(j) \le e^{-n\gamma}
g_t(k)$ for all $0\le j, k\le N-1$, $t\ge 0$, which together with
$m_N(j)\geq e^{\gamma}$, $0\le j\le N-1$, gives that $\Vert
g_t\Vert^2_M \le e^{-(2n+1)\gamma} N^{-1} \Vert g_t \Vert^2_2$.  By
Lemma \ref{s06}, $\Vert g_t \Vert^2_2 \le e^{C_0 t} \Vert g_0
\Vert^2_2$.  In conclusion,
\begin{equation*}
f_t(j)^2 \;\le\; C_0 e^{C_0 t} N^{-1} \Vert g_0
\Vert^2_2 \;\le\; C_0 e^{C_0 t} \Vert g_0 \Vert^2_M \;=\;
C_0 e^{C_0 t} \Vert f_0 \Vert^2_M\;,
\end{equation*}
which proves the lemma.
\end{proof}

Fix $n\ge 1$ and denote by $q_t (j, \cdot) = q^{(n)}_t(j,\cdot)$ the
solution of the linear equation \eqref{fa1} with initial condition
$q_0(j,k) = \delta_{j,k}$.  Fix a function $f: \Lambda_N \to \bb
R$. We may represent the solution $f_t$ of \eqref{fa1} with initial
condition $f$ as $f_t(k) = \sum_{j\in \Lambda_N} f(j) q_t(j,k)$. In the
particular case where $f(k)=1$ for all $k\in \Lambda_N$, by Lemma
\ref{s05}, 
\begin{equation*}
\max_{k\in \Lambda_N} \sum_{j\in \Lambda_N} q_t(j,k) \;=\; 
\max_{k\in \Lambda_N} f_t(k) \;\le\; C_0 e^{C_0 t}\;. 
\end{equation*}

\begin{lemma}
\label{s08}
Fix $n\ge 1$, a strictly positive initial condition $f_0: \Lambda_N
\to \bb R$ and let $f_t$ be the solution of \eqref{fa1}.  For every
$T>0$, there exists a positive constant $c_0$, depending only on
$f_0$, $E$, $\alpha$, $\beta$ and $T$, such that
\begin{equation*}
c_0 \;\le\; f_t(j)
\end{equation*}
for all $0\le t\le T$, $j\in \Lambda_N$.
\end{lemma}

\begin{proof}
By the maximum principle, it is enough to prove the lemma for a
constant initial profile. Assume, therefore, that $f_0(j)=a$ for all
$j\in \Lambda_N$ and for some $a>0$. A simple computation shows that
\begin{equation*}
\frac d{dt} \frac 1N \sum_{j=0}^{N-1} f_t(j) \, m_N(j)
\;=\;  \frac 1N \sum_{j=0}^{N-1} (\Omega_n f_t)(j) \, m_N(j)
\;\geq \; -\, \alpha R_n  f_t(0) \, m_N(0)\;.
\end{equation*}
By Lemma \ref{s02}, $f_t(0) \le N^{-1} \sum_{0\le j\le N-1}
f_t(j)$. On the other hand, $m_N(0) = 1 \le m_N(j) e^{-\gamma}$ for
all $j\in \Lambda_N$. Hence,
\begin{equation*}
\frac d{dt} \frac 1N \sum_{j=0}^{N-1} f_t(j) \, m_N(j)
\;\ge\; -\, \alpha R_n  e^{-\gamma} \frac 1N \sum_{j=0}^{N-1} f_t(j) \,
m_N(j)\;. 
\end{equation*}
Therefore, by Gronwall's inequality and since $R_n$ is bounded above
by a finite constant independent of $N$,
\begin{equation*}
\frac 1N \sum_{j =0}^{ N-1} f_t(j) \, m_N(j) \;\ge\; 
e^{-At} \frac 1N \sum_{ j=0}^{N-1} f_0(j) \, m_N(j) 
\; \geq \; a e^{\gamma}e^{-At}.  
\end{equation*}
A constant profile satisfies both conditions of Lemma \ref{s02}. We
may therefore apply this lemma to bound above $N^{-1} \sum_{j\in
  \Lambda_N} f_t(j)$ by $C_0 \min_{k\in \Lambda_N} f_t(k)$, which
completes the proof since $ m_N(j)\le 1$.
\end{proof}

Next result provides a bound for the fundamental solution of
\eqref{fa1}. The proof is based on the classical arguments of
hypercontractivity \cite{da1, ds96}. We need, however, to estimate
additional terms which appear because $\Omega_n$ is not a generator. 

For $\epsilon >0$, let $\delta = \epsilon/(1+\epsilon)$,
and let $\varphi_\epsilon : [0,1] \to [\delta, 1-2\epsilon]$ be given
by
\begin{equation*}
\varphi_\epsilon (t) \;=\; 
\begin{cases}
\sqrt{\delta^2 + t}\;, & \text{for $0\le t\le 1/8$,} \\
1 - \sqrt{4 \epsilon^2 + 1-t}\;, & \text{for $7/8\le t\le 1$.}
\end{cases}
\end{equation*}
We complete the definition of $\varphi_\epsilon$ in the interval
$[1/8, 7/8]$ in a way to obtain an increasing $C^1$ function whose
derivative in the interval $[1/8, 7/8]$ is bounded by $2$. Note that
this bound is compatible with $\varphi_\epsilon'(1/8)$ and
$\varphi_\epsilon'(7/8)$, which are both bounded by $\sqrt{2}$.

Actually, the exact form of $\varphi_\epsilon$ is irrelevant for the
proof of Lemma \ref{s04}. The only important properties needed are
that
\begin{equation*}
\int_0^1 \frac 1{\varphi_\epsilon(t) [1-\varphi_\epsilon(t)]} \, dt 
\;<\; \infty\;, \quad\text{and}\quad 
\int_0^1 \dot \varphi_\epsilon(t)  \, \log 
\frac{ \dot \varphi_\epsilon(t)}{\varphi_\epsilon(t) [1-\varphi_\epsilon(t)]} \, dt \;<\;
\infty\; ,
\end{equation*}
where $\dot \varphi_\epsilon(t)$ represents the derivative of
$\varphi_\epsilon$. 

\begin{lemma}
\label{s04}
Fix $n\ge 1$ and recall that we denote by $q_t (j, \cdot)$ the
solution of the linear equation \eqref{fa1} with initial condition
$q_0(j,k) = \delta_{j,k}$.  Assume that $N\ge n+1$ and let $A_1=
-\gamma n \beta$. There exists a finite constants $C_0$, depending
only on $E$, $\beta$ and $n$, such that
\begin{equation*}
\max_{0\le j,k\le N-1} q_T(j,k) \;\le\; \frac {C_0 e^{C_0
    T}}{\sqrt{ N^2 T }}
\end{equation*}
for all $T$ such that
\begin{equation}
\label{fa3}
\log (TN^2) \;\ge\; 16\;, \quad
\log (TN^2) \le\; \sqrt{\frac {TN^2}{8 A_0}} \;,\quad
\log (TN^2) \;\le\; N \Big( 1 \wedge \frac 1{8e^E A_1}\Big) \;.
\end{equation}
where $A_0$ is given in \eqref{fa2}.
\end{lemma}

\begin{proof}
Here we follow \cite{l1,l2}. In this proof $C_0$ represents a finite
constant depending only on $\beta$, $E$ and $n$, which may change from
line to line.

Fix $0\le k\le N-1$ and $T$ in the range \eqref{fa3}.  Let
$\epsilon^{-1} = \log (TN^2)$, $p: [0,T] \to [1+ \epsilon ,
2\epsilon^{-1}]$ be given by $p(t) = [1-\varphi_\epsilon
(t/T)]^{-1}$. Set $f_t(\cdot) = q_t(x,\cdot)$, $u_t^2 = f_t^{p(t)}$,
$v_t^2 = u_t^2 / \Vert u_t \Vert_2^2$. An elementary computation,
identical to the one presented at the beginning of the proof of
Theorem 2.1 in \cite{l1}, gives that
\begin{equation}
\label{fa5}
\frac{d}{dt} \log \|f_t\|_{p(t)} \;\le \;
\frac{\dot p(t)}{p(t)^2} \int v_t^2 \log v_t^2\,dm_N
- \frac{2[p(t)-1]}{p(t)^2}\, D_N(v_t) \;+\; A_1 N v_t(N-1)^2\;.
\end{equation}

Set
\begin{equation*}
\ell (t)^2 \;=\; N^2 \Big\{ \frac {p(t) -1}{4 A_0 \dot p(t)} \wedge 1
\Big\} \;=\; \frac {T N^2}{A_0}
\Big\{ \frac {\varphi_\epsilon(t/T) [1-\varphi_\epsilon(t/T)]}{ 4 \dot
  \varphi_\epsilon(t/T)} \wedge \frac{A_0}T \Big\} \;.
\end{equation*}
By the second condition in \eqref{fa3}, $\ell(t) \ge 1$.  Divide the
interval $\Lambda_N$ in subintervals of length $\ell(t)$. The last
interval has length between $\ell(t)$ and $2\ell(t)-1$.  By the
logarithmic Sobolev inequality \eqref{fa2} and by the the proof of
Lemma 4.3 of \cite{l1}, since $m_N(k) \ge e^\gamma$, the first term on
the right hand side of \eqref{fa5} is less than or equal to
\begin{equation*}
\frac{\dot p(t)}{p(t)^2} \Big\{ A_0 \frac {4 \ell(t)^2}{N^2} D_N(v_t) \;
-\; \log [ e^\gamma \ell (t)] \Big\}\; \cdot
\end{equation*}
By definition of $\ell (t)$, the right hand side of \eqref{fa5} is
bounded by
\begin{equation}
\label{fa4}
- \frac{\dot p(t)}{2 p(t)^2} \log [e^{2\gamma}\ell (t)^2]
\;-\; \frac{[p(t)-1]}{p(t)^2}\, D_N(v_t) \;+\; A_1 N v_t(N-1)^2\;.
\end{equation}

Let
\begin{equation*}
m(t) \;=\; N \frac{ p(t)-1}{p(t)^2} \Big\{
\frac 1 {2 e^E A_1} \wedge 4 \Big\}
\;=\; N  \varphi_\epsilon(t/T) [1-\varphi_\epsilon(t/T)] \Big\{
\frac 1 {2 e^E A_1} \wedge 4 \Big\}\;.
\end{equation*}
Notice that $m(t) \le N$, because $0\le p(t)^{-1} \le 1$. On the other
hand, as $p(t)^{-1}[1-p(t)^{-1}] \ge \{4 \log (TN^2)\}^{-1}$ and $N
\ge \log (TN^2) \{ 8 e^E A_1 \vee 1\}$, we have that $m(t) \ge
1$. Adding and subtracting the average of $v_t(j)$ over the interval
$\{N-m(t), \dots , N-1\}$, and repeating the same argument as in the
proof of Lemma \ref{s06}, since $-\gamma \le E$, we obtain that
\begin{equation*}
\begin{split}
v_t(N-1)^2&\leq
2 m(t) \sum_{j=N-m(t)}^{N-2} \{ v_t(j+1) - v_t(j)\}^2 \; +\; \frac
2{m(t)} \sum_{j=N-m(t)}^{N-1} v_t(j)^2 \\
 &\le\;  \frac{2 e^E m(t)}{N^2} D_N(v_t) \;+\; \frac {2 e^E}{ m(t)}
\end{split}
\end{equation*}
because $\Vert v(t)\Vert_2 =1$. By definition of $m(t)$, the first
term of this expression multiplied by $A_1 N$ may be absorbed by the
Dirichlet form in \eqref{fa4}. Hence, \eqref{fa4} is less than or
equal to
\begin{equation*}
- \frac{\dot p(t)}{2 p(t)^2} \log [e^{2\gamma} \ell (t)^2]
\;+\; C_0 \frac{p(t)^2}{p(t)-1}\; \cdot
\end{equation*}

Up to this point, we proved that
\begin{equation}
\label{fa6}
\log \Big(\frac{\|f_T\|_{p(T)}}{\|f_0\|_{p(0)}} \Big)\;\le \;
-  \int_0^T \frac{\dot p(t)}{2 p(t)^2} \log [e^{2\gamma} \ell (t)^2] \, dt
\;+\; C_0 \int_0^T \frac{p(t)^2}{p(t)-1}\, dt \; \cdot
\end{equation}
Since $\dot p(t)/p(t)^2 = T^{-1} \dot \varphi_\epsilon (t/T)$, in view
of \eqref{fa3}, the first term on the right hand side is less than or
equal to
\begin{equation*}
-\; \frac 12 \log (TN^2) \;+\; C_0 \;+\; \frac 12 \int_0^1 \dot
\varphi_\epsilon (t) \log \Big\{ \frac {\dot  \varphi_\epsilon(t)}
{\varphi_\epsilon(t) [1-\varphi_\epsilon(t)]} \vee \frac T{4A_0}
\Big\} \, dt\;. 
\end{equation*}
Since $\log (a \vee b) \le \log_+ a + \log_+ b$, where $\log_+ a =
\log a \vee 0$, the previous integral can be estimated by the sum of
two terms. The first one is $\log_+ (T/4A_0) \le C_0 T$, while the
second one is
\begin{equation*}
\frac 12 \int_0^1 \dot \varphi_\epsilon (t) 
\log_+ \Big\{ \frac {\dot  \varphi_\epsilon(t)}
{\varphi_\epsilon(t) [1-\varphi_\epsilon(t)]} \Big\} \, dt\;. 
\end{equation*}
On the interval $[1/8,7/8]$, $\dot \varphi_\epsilon(t)$ is bounded by
$2$ and $\varphi_\epsilon(t) [1-\varphi_\epsilon(t)]$ is bounded below
by a positive constant independent of the parameters. On the other
hand, on the interval $[0,1/8]$, in view of \eqref{fa3}, $\dot
\varphi_\epsilon(t)/\{\varphi_\epsilon(t) [1-\varphi_\epsilon(t)]\}
\ge [\delta^2 + t]^{-1} \ge 1$. Hence, in this interval, the previous
integral is bounded by
\begin{equation*}
\frac 14 \int_0^{1/8} \frac 1{\sqrt{\delta^2 + t}}  
\log \frac 1 {\delta^2 + t} \, dt \;\le\; C_0\;. 
\end{equation*}
A similar analysis can be carried out in the interval $[7/8,1]$.

The second term on the right hand side of \eqref{fa6} is equal to
\begin{eqnarray*}
C_0 T \, \int_0^1 \frac{1}{\varphi_\epsilon(t) [1-\varphi_\epsilon(t)]} \,
dt \;\le\; C'_0 T\;. 
\end{eqnarray*}
Therefore,
\begin{equation*}
\log\Big( \frac{\|f_T\|_{p(T)}}{\|f_0\|_{p(0)}} \Big)\;\le \;
-\; (1/2) \log \{ N^2 T \} \;+\; C_0 \; +\; C_0 T \;.
\end{equation*}
To conclude the proof of the lemma, it remains to observe that $\Vert
f_T \Vert_M \le e^{E/2}$ $\Vert f_T \Vert_{p(T)}$, $\Vert f_0
\Vert_{p(0)} \le 1$.
\end{proof}

\begin{corollary}
\label{s03}
Fix $n\ge 1$, $T_0>0$, and denote by $q_t (j, \cdot)$ the solution of
the linear equation \eqref{fa1} with initial condition $q_0(j,k) =
\delta_{j,k}$.  There exist a finite constant $C_0$ and $N_0\ge 1$,
depending only on $E$, $\beta$ and $n$, such that
\begin{equation*}
q_t(j,k) \;\le\; \frac {C_0 e^{C_0 t}}
{\sqrt{ N^2 t}}
\end{equation*}
for all $0\le t\le T_0$, $N\ge N_0$, and $0\le j,k\le N-1$.
\end{corollary}

\begin{proof}
Fix $n\ge 1$, $T_0>0$, and $0\le j\le N-1$. There exists $N_0\ge n+1$
for which the last condition in \eqref{fa3} is satisfied for all
$0\le t\le T_0$, $N\ge N_0$. 

There exists $a>0$ such that $\sup_{x\ge a} \log x/\sqrt{x} \le
(8A_0)^{-1/2}$. Let $b = \max\{a , e^{16}\}$.  Fix $0\le t\le T_0$.
If $tN^2\le b$, by Lemma \ref{s05},
\begin{equation*}
\max_{0\le k\le N-1} q_t(j,k)  \;\le\; C_0 e^{C_0 t} \;\le\;
\frac {\sqrt{b} \, C_0 e^{C_0 t}} {\sqrt{ N^2 t }} \;\cdot
\end{equation*}
On the other hand, if $tN^2\ge b$, $t$ fulfills all the assumptions of
the previous lemma.  This ends the proof.
\end{proof}

We conclude this section with a remark used several times in the
previous sections. Let $f_t(k) = \sum_{j\in \Lambda_N}
q_t(j,k)$. Thus, $f$ is the solution of \eqref{fa1} with initial
condition $f(k)=1$ for all $k\in\Lambda_N$. By Lemma \ref{s05},
for all $T>0$, there exists a finite constant $C_0$, depending only on
$E$, $\beta$ and $T$ such that
\begin{equation}
\label{f2.3}
\sup_{0\le t\le T} \max_{k\in\Lambda_N} \sum_{j\in \Lambda_N}
q_t(j,k) \;=\; \sup_{0\le t\le T} \max_{k\in\Lambda_N} f_t(k) \;\le\;
C_0\; .
\end{equation}

\smallskip
\noindent{\bf Acknowledgements.}
P. G. thanks CNPq (Brazil) for support through the research project
``Additive functionals of particle systems'', Universal
n. 480431/2013-2, also thanks FAPERJ ``Jovem Cientista do Nosso Estado''
for the grant E-25/203.407/2014 and the Research Centre of Mathematics
of the University of Minho, for the financial support provided by
``FEDER'' through the ``Programa Operacional Factores de Competitividade
COMPETE'' and by FCT through the research project
PEst-OE/MAT/UI0013/2014.


\begin{thebibliography}{99}

\bibitem{bdgjl02} L. Bertini, A.~De~Sole, D.~Gabrielli,
  G.~Jona-Lasinio, C.~Landim: Macroscopic fluctuation theory for
  stationary non-equilibrium states.  J. Statist. Phys. \textbf{107},
  635--675 (2002).

\bibitem{blm1} L. Bertini, C.  Landim, M. Mourragui; Dynamical large
  deviations for the boundary driven weakly asymmetric exclusion
  process. Ann. Probab. {\bf 37}, 2357-2403 (2009).

\bibitem{br1} T. Brox, H. Rost: Equilibrium fluctuations of stochastic
  particle systems: the role of conserved
  quantities. Ann. Probab. {\bf 12}, 742-759 (1984)

\bibitem{da1}  E. B. Davies: \emph{Heat Kernels and Spectral
  Theory}. Cambridge Univ. Press. 1989.

\bibitem{DMPS} A. De Masi, E. Presutti, E. Scacciatelli: The weakly
  asymmetric simple exclusion process.  Ann. Inst. H. Poincar\'e
  Probab. Statist., {\bf{ 25}}, no. 1, 1--38 (1989).

\bibitem{delo1} B. Derrida, C. Enaud, C. Landim, S. Olla: Fluctuations
  in the weakly asymmetric exclusion process with open boundary
  conditions. J. Stat. Phys. {\bf 118}, 795-811 (2005)

\bibitem{dls} B. Derrida, J.~L. Lebowitz, and E.~R. Speer: Large
  deviation of the density profile in the steady state of the open
  symmetric simple exclusion process, J. Statist. Phys. \textbf{107},
  599--634 (2002).

\bibitem{ds96} P. Diaconis, L. Saloff-Coste: Logarithmic Sobolev
  inequalities for finite Markov chains.  Ann. App. Probab.  {\bf 6},
  695--750 (1996).

% \bibitem{d} Dittrich, P.: Travelling waves and long-time behaviour of
%   the weakly asym\-met\-ric exclusion process. Probab. Th. Rel. Fields
%   {\bf 86}, 443--455, (1990).

\bibitem{dg} P. Dittrich, J. G\"artner: A central limit theorem for
  the weakly asymmetric simple exclusion process. Math. Nachr.  {\bf
    151}, 75--93 (1991).

\bibitem{ed1} C. Enaud, B. Derrida: Large Deviation Functional of the
  Weakly Asymmetric Exclusion Process. J. Stat. Phys. {\bf 114},
  537--562, (2004)

\bibitem{flm1} J. Farfan, C. Landim, M. Mourragui; Hydrostatics and dynamical
  large deviations of boundary driven gradient symmetric exclusion
  processes. Stoch. Process. Appl. {\bf 121}, 725--758 (2011).

\bibitem{g} J. G{\"a}rtner: Convergence towards {B}urger's equation
  and propagation of chaos for weakly asymmetric exclusion processes.
  Stoch. Proc. Appl. {\bf 27}, 233--260 (1988).

% \bibitem{HS} Holley, R. and Stroock., D.: Generalized
%   {O}rnstein-{U}hlenbeck processes and infinite particle branching
%   {B}rownian motions. Publ. Res. Inst. Math. Sci., {\bf 14} (3),
%   741--788, (1978).

% \bibitem{glm2} P. Gon\c{c}alves, C.  Landim, A. Milanes: Stationary
%   fluctuations of one-dimensional boundary driven weakly asymmetric
%   exclusion processes. in preparation (2015).

\bibitem{js} J. Jacod, A. N. Shiryaev: {\emph Limit theorems for
    stochastic processes}. Second edition. Grundlehren der
  Mathematischen Wissenschaften [Fundamental Principles of
  Mathematical Sciences], {\bf 288}. Springer-Verlag, Berlin, 2003. 

\bibitem{kl} C. Kipnis, C. Landim; {\it Scaling Limits of Interacting
    Particle Systems}, Grundlheren der mathematischen Wissenschaften
  {\bf 320}, Springer-Verlag, Berlin, New York, 1999.

\bibitem{l1} C. Landim: Decay to equilibrium in $L^\infty$ of finite
  interacting particle systems in infinite volume. Markov
  Proc. Rel. Fields {\bf 4}, 517--534 (1998).

\bibitem{l2} C. Landim: Gaussian estimates for symmetric simple
  exclusion processes. Ann. Fac. Sci. Toulouse Math. {\bf 14},
  683--703 (2005).

\bibitem{lmo1} C.  Landim, A. Milanes, S. Olla: Stationary and
  nonequilibrium fluctuations in boundary driven exclusion processes.
  Markov Proces.  Related Fields {\bf 14}, 165--184 (2008).


% \bibitem{KOV} Kipnis, C., Olla, S. and Varadhan, S.: Hydrodynamics and
%   large deviation for simple exclusion processes. Comm. Pure
%   Appl. Math. \textbf{42} (2), 115--137 (1989).

% \bibitem{Mit} Mitoma, I.: Tightness of probabilities on
%   $C([0,1];{\mathcal{S}'})$ and
%   $D([0,1];\mathcal{S}')$. Ann. Probab. \textbf{11} (4), 989--999
%   (1983).

\bibitem{rs} M. Reed, B. Simon: Methods of Modern Mathematical
  Physics, vol 4. Analysis of operators, (1978).


\end{thebibliography}
\end{document}